\documentclass[12pt,draft]{amsart}
\usepackage{setspace,units}

\newtheorem{theorem}{Theorem}[section]
\newtheorem{lemma}[theorem]{Lemma}

\newtheorem{Proposition}[theorem]{Proposition}
\theoremstyle{definition}
\newtheorem{definition}[theorem]{Definition}

\newtheorem{remark}[theorem]{Remark}
\newtheorem{Conjecture}[theorem]{Conjecture}

\numberwithin{equation}{section}

\RequirePackage{amsmath}
\RequirePackage{amssymb}
\RequirePackage{amsthm}
 \RequirePackage{epsfig}

\begin{document}

\date{}
\title[Duffin-Schaeffer conjecture]
{The Duffin-Schaeffer type conjectures in various local fields}

\author{Liangpan Li}

\address{Department of Mathematics,
Loughborough University, LE11 3TU, UK}
 \email{L.Li@lboro.ac.uk, liliangpan@gmail.com}



\date{}

\begin{abstract}
This paper discovers a new phenomenon about the Duffin-Schaeffer
conjecture, which claims that
$\lambda(\cap_{m=1}^{\infty}\cup_{n=m}^{\infty}{\mathcal E}_n)=1$ if
and only if $\sum_n\lambda({\mathcal E}_n)=\infty$, where $\lambda$
denotes the Lebesgue measure on $\mathbb{R}/\mathbb{Z}$,
\[
{\mathcal E}_n={\mathcal E}_n(\psi)=\bigcup_{m=1 \atop
(m,n)=1}^n\big(\frac{m-\psi(n)}{n},\frac{m+\psi(n)}{n}\big),
\]
$\psi$ is any non-negative arithmetical function. Instead of
studying $\cap_{m=1}^{\infty}\cup_{n=m}^{\infty}{\mathcal E}_n$ we
introduce an even fundamental object $\cup_{n=1}^{\infty}{\mathcal
E}_n$ and conjecture there exists a universal constant $C>0$ such
that
\[\lambda(\bigcup_{n=1}^{\infty}{\mathcal E}_n)\geq C\min\{\sum_{n=1}^{\infty}\lambda({\mathcal E}_n),1\}.\]
It is shown that this conjecture is equivalent to the
Duffin-Schaeffer conjecture. Similar phenomena are found in the
fields of $p$-adic numbers and formal Laurent series. As a
byproduct, we  answer conditionally a question of Haynes by showing
that one can always use the quasi-independence on average method to
deduce $\lambda(\cap_{m=1}^{\infty}\cup_{n=m}^{\infty}{\mathcal
E}_n)=1$ as long as the Duffin-Schaeffer conjecture is true. We also
show among several others that two conjectures of Haynes, Pollington
and Velani are equivalent to the Duffin-Schaeffer conjecture, and
introduce for the first time a weighted version of the second
Borel-Cantelli lemma
   to the study of the Duffin-Schaeffer conjecture.
\end{abstract}

\maketitle


\section{Introduction to the Duffin-Schaeffer conjecture}

Throughout the paper we use the following notations:
\begin{itemize}
\item $p$ denotes a prime number,
\item  $n,h$ denote positive integers,
\item $\varphi(n)$ denotes the Euler phi function, \item $\lambda$ denotes the
Lebesgue measure on $\mathbb{R}/\mathbb{Z}$, \item $f(x)\nearrow A$
($\searrow A$) means $f(x)$ tends increasingly (decreasingly) to
$A$, \item $f(x)\ll g(x)$ or $g(x)\gg f(x)$ means $|f(x)|\leq
C|g(x)|$ for some universal constant $C>0$, $f(x)\asymp g(x)$ means
both $f(x)\ll g(x)$ and $g(x)\ll f(x)$,
\item $B(x,r)$ ($\overline{B}(x,r)$) denotes the open (closed) ball with center  $x$ and radius $r$
in a given metric space.
\end{itemize}

In this section we will introduce case by case the Duffin-Schaeffer
type conjectures in the fields $\mathbb{R}$ of real numbers,
$\mathbb{Q}_p$ of $p$-adic numbers, and $\mathbb{F}((X^{-1}))$ of
formal Laurent series.

\subsection{Duffin-Schaeffer conjecture}

For any non-negative function $\psi:\mathbb{N}\rightarrow\mathbb{R}$
and any positive integer $n$, we define ${\mathcal
E}_n(\psi)\subseteq \mathbb{R}/\mathbb{Z}$ by
\begin{equation}
{\mathcal E}_n={\mathcal E}_n(\psi)=\bigcup_{m=1 \atop
(m,n)=1}^n\big(\frac{m-\psi(n)}{n},\frac{m+\psi(n)}{n}\big).
\end{equation}
Let $W(\psi)$ denote the collection of points
$x\in\mathbb{R}/\mathbb{Z}$  which fall in infinitely many of the
sequence of the sets $\{{\mathcal E}_n\}_{n\in\mathbb{N}}$, that is,
\begin{equation}
W(\psi)=\limsup_{n\rightarrow\infty}{\mathcal E}_n=\bigcap_{m=1}^{\infty}\bigcup_{n=m}^{\infty}{\mathcal E}_n.
\end{equation}
The Lebesgue measure of ${\mathcal E}_n$ is obviously bounded above
by $\frac{2\psi(n)}{n}\varphi(n)$. According to the first
Borel-Cantelli lemma, we see that if the series
\begin{equation}
\label{formula 11} \sum_{n=1}^{\infty}\frac{\psi(n)}{n}\varphi(n)
\end{equation} is convergent,
then $W(\psi)$ is of zero Lebesgue measure. In 1942 Duffin and
Schaeffer (\cite{DuffinSchaeffer}) proposed the following conjecture
in metric number theory:

\begin{Conjecture}\label{conj 11}  If
(\ref{formula 11}) is divergent, then $\lambda(W(\psi))=1$.
\end{Conjecture}

Although various partial results are known (see \cite{Harman98} for
details and  references before 2000 and
\cite{Aistleitner,BHHV,Haynes,Li} for recent progresses), the full
conjecture represents one of the most fundamental unsolved problems
in metric number theory.

The main purpose of this paper is to establish various equivalent
forms for the original Duffin-Schaeffer conjecture. We begin with a
classical estimate of Pollington and Vaughan (\cite[formula
(3)]{PV2}) claiming that if $\psi(n)\leq\frac{n}{2\varphi(n)}$, then
$\lambda({\mathcal E}_n)\geq\frac{\psi(n)\varphi(n)}{n}$. This
implies unconditionally that
 \begin{equation}\label{formula 141414}
\min\{\frac{\psi(n)\varphi(n)}{n},\frac{1}{2}\}\leq\lambda({\mathcal
E}_n)\leq\min\{2\frac{\psi(n)\varphi(n)}{n},1\}.
\end{equation}
Instead of studying $W(\psi)$ we introduce an even fundamental
object
 \begin{equation}
 Z(\psi)=\bigcup_{n=1}^{\infty}{\mathcal
E}_n,
\end{equation}
whose Lebesgue measure is bounded above by
$\min\{2\sum_{n=1}^{\infty}\frac{\psi(n)\varphi(n)}{n},1\}$. Similar
to (\ref{formula 141414}) we conjecture this trivial upper bound is
essentially a non-trivial lower bound despite some loss of constant:

\begin{Conjecture}\label{Conjecture 16} There exists a universal
constant $C>0$ such that for any non-negative function $\psi$,
\begin{equation}\lambda(Z(\psi))\geq C\min\{\sum_{n=1}^{\infty}\frac{\psi(n)\varphi(n)}{n},1\}.\end{equation}
\end{Conjecture}

At this stage (see also the end of Section \ref{dsc2}) it is hard to
convince the readers why the above conjecture is possibly true, but
we can show

\begin{theorem}\label{theorem 17}
Conjecture \ref{Conjecture 16} and the Duffin-Schaeffer conjecture
are  equivalent.
\end{theorem}

Similar phenomena will be confirmed in the fields of  $p$-adic
numbers and formal Laurent series. At the moment we concentrate on
the classical case.

At the end of the paper \cite{Haynes p-adic} Haynes asked  whether
there exists a non-negative function $\psi$ for which one cannot use
the quasi-independence on average method (see the last part of this
section for an introduction) to deduce $\lambda(W(\psi))=1$.  As a
byproduct of Theorem \ref{theorem 17}, we can give a conditional but
almost best possible answer: If the Duffin-Schaeffer conjecture is
true, then the answer is NO.

Next let us recall three related conjectures due to Sanju Velani and
his coauthors. Letting $f:\mathbb{R}^{+}\rightarrow\mathbb{R}^{+}$
be a dimension function, and letting ${\mathcal H}^{f}$ be the
corresponding Hausdorff $f$-measure on $\mathbb{R}/\mathbb{Z}$,
Beresnevich and Velani (\cite[Conejcture 2]{BV06})  proposed the
following Hausdorff measure version of the Duffin-Schaeffer
conjecture:

\begin{Conjecture}\label{conj 12}   Let $f$ be a dimension function such that $r^{-1}f(r)$ is monotonic. If
\begin{equation}\label{formula
12}\sum_{n=1}^{\infty}f(\frac{\psi(n)}{n})\varphi(n)\end{equation}
is divergent, then ${\mathcal H}^{f}(W(\psi))={\mathcal
H}^{f}(\mathbb{R}/\mathbb{Z})$.
\end{Conjecture}

The original statement of Conjecture \ref{conj 12} is in fact a
$k$-dimensional analogue, which has already been confirmed
 for  $k\geq2$ (\cite[Corollary 1]{BV06}). As a consequence of their  Mass
 Transference Principle (\cite[Theorem 2]{BV06}), Beresnevich and
Velani  showed   (\cite[Theorem 1]{BV06}) that Conjecture \ref{conj
12} is equivalent to the Duffin-Schaeffer conjecture.

Recently, Haynes, Pollington and Velani (\cite[Conjectures 1 \&
2]{Haynes}) proposed the following ``weakening" versions of the
Duffin-Schaeffer conjecture by assuming extra divergence:

\begin{Conjecture}\label{Conjecture 13}
Let $f$ be a dimension function such that $r^{-1}f(r)\nearrow\infty$
 as $r\rightarrow0$. If (\ref{formula 11}) is divergent, then ${\mathcal H}^f(W(\psi))=\infty$.
\end{Conjecture}

\begin{Conjecture}\label{Conjecture 14}
Let $f:[0,\infty)\rightarrow\mathbb{R}$ be an increasing
non-negative function such that $r^{-1}f(r)\rightarrow0$ as
$r\rightarrow0$. If (\ref{formula 12}) is divergent, then
$\lambda(W(\psi))=1$.
\end{Conjecture}

The proposers also believe that the latter two conjectures are in
principle easier to establish than the Duffin-Schaeffer conjecture,
but we can show

\begin{theorem}\label{theorem 15} Conjecture \ref{Conjecture 13}, Conjecture \ref{Conjecture
14} and the Duffin-Schaeffer conjecture are all equivalent.
\end{theorem}

Nevertheless, the interested readers may still study Conjectures
\ref{Conjecture 13} \& \ref{Conjecture 14} for some particularly
chosen functions to get some partial results (see \cite[Problems 1
\& 2]{Haynes}). The main tool for proving Theorem \ref{theorem 15}
is
 the well-known fact that there is
no fastest converging or slowest diverging series (see e.g.
\cite{Rudin}).

 To summarize, we have the equivalence between the
Duffin-Schaeffer conjecture for $W(\psi)$, the newest Lebesgue
measure version Conjecture \ref{Conjecture 16} for $Z(\psi)$ as well
as the Hausdorff measure version Conjecture \ref{conj 12} for
$W(\psi)$. Conjecture \ref{Conjecture 16} might be more easier to
attack than Conjecture \ref{conj 12} as it is a standard Lebesgue
measure statement rather than a Hausdorff measure one.

\subsection{$p$-adic approximation}
 For any prime $p$, let  $\mathbb{Q}_p$ denote the
 field of $p$-adic numbers with absolute value $|\cdot|_p$, and let
 $\mathbb{Z}_p$ denote the ring of integers
\[\mathbb{Z}_p=\{x\in\mathbb{Q}_p:|x|_p\leq1\}.\]
Since $\mathbb{Q}_p$ is a locally compact topological group under
addition, there exists a unique Haar measure $\mu_p$ on
$\mathbb{Q}_p$ such that $\mu_p(\mathbb{Z}_p)=1$. For any
non-negative function $\psi:\mathbb{N}\rightarrow\mathbb{R}$ and any
positive integer $n$, we define
\begin{equation}{\mathcal K}_n(\psi)=\bigcup_{a=-n\atop (a,n)=1}^n\Big\{x\in\mathbb{Z}_p:\big|x-
\frac{a}{n}\big|_p\leq\frac{\psi(n)}{n}\Big\},\end{equation} and set
\begin{equation}W_p(\psi)=\limsup_{n\rightarrow\infty}{\mathcal K}_n(\psi)=\bigcap_{m=1}^{\infty}\bigcup_{n=m}^{\infty}{\mathcal K}_n(\psi).
\end{equation}

Recently, Haynes (\cite{Haynes p-adic}) studied in detail the metric
Diophantine approximation in $\mathbb{Q}_p$ by attacking the next
Duffin-Schaeffer type conjecture (\cite[Conjecture 1]{Haynes
p-adic}):

\begin{Conjecture}\label{conj710}
$\mu_p(W_p(\psi))=1$ if and only if
$\sum_{n\in\mathbb{N}}\mu_p({\mathcal K}_n(\psi))=\infty$.
\end{Conjecture}

 Similar to
Conjecture \ref{Conjecture 16} we propose

\begin{Conjecture}\label{conj711} There exists a universal constant $C>0$ depending only on $p$  such that
for any non-negative function
$\psi:\mathbb{N}\rightarrow\mathbb{R}$,
\begin{equation}
\mu_p(\bigcup_{n=1}^{\infty}{\mathcal K}_n(\psi))\geq
C\min\{\sum_{n=1}^{\infty}\mu_p({\mathcal K}_n(\psi)),1\}.
\end{equation}
\end{Conjecture}

Similar to Theorem \ref{theorem 17} we can show
\begin{theorem}\label{theorem 712} Conjecture \ref{conj710} and Conjecture \ref{conj711}
are equivalent.
\end{theorem}

The proof of Theorem \ref{theorem 712} is in essence similar to that
of Theorem \ref{theorem 17}, but still needs to be treated
independently.

We will also show that one can always use the quasi-independence on
average method to deduce $\mu_p(W_p(\psi))=1$ as long as the
$p$-adic version of the Duffin-Schaeffer conjecture is true.

Several recent progresses on the classical Duffin-Schaeffer
conjecture (\cite{Aistleitner,BHHV,Li}) can be naturally transferred
into new results on the $p$-adic version of the Duffin-Schaeffer
conjecture via a lemma of Haynes (\cite[Lemma 3]{Haynes p-adic}),
but we will not pursue this direction in the paper.

\subsection{Formal Laurent series}

Let $\mathbb{F}$ be a finite field of $q$ elements. Throughout we
will use the following notations which play the roles of integers,
rational numbers, real numbers, $\mathbb{R}/\mathbb{Z}$, and the
absolute value, respectively:

\begin{itemize}
\item $\mathbb{F}[X]$ denotes the set of polynomials with
 $\mathbb{F}$-coefficients,

\item $\mathbb{F}(X)$ denotes the fraction field of $\mathbb{F}[X]$,

\item $\mathbb{F}((X^{-1}))$ denotes the set of formal Laurent
series,

\item$\mathbb{L}$ denotes the set of elements of $\mathbb{F}((X^{-1}))$ with
degrees less than zero,

\item $|f|\triangleq q^{\partial f}$, where $\partial f$ denotes
the degree of $f\in\mathbb{F}((X^{-1}))$.
\end{itemize}
The metric $\rho$ on $\mathbb{F}((X^{-1}))$ is naturally defined as
$\rho(f,g)=|f-g|$. Since $\mathbb{F}((X^{-1}))$ is a locally compact
topological group under addition, there exists a unique Haar measure
$\nu$ on $\mathbb{F}((X^{-1}))$ such that $\nu(\mathbb{L})=1$.  The
$d$-fold ($d\in\mathbb{N}$) product of measure $\nu$ on
$\mathbb{F}((X^{-1}))^d$ is denoted by $\nu_d$.

For any non-negative function
$\Psi:\mathbb{F}[X]\rightarrow\mathbb{R}$ and any monic
$Q\in\mathbb{F}[X]$,  we define
\begin{equation}
{\mathcal E}_Q(\Psi)=\bigcup_{{P\in\mathbb{F}[X]\atop \partial
P<\partial Q}\atop
(P,Q)=1}\big\{f\in\mathbb{L}:|f-\frac{P}{Q}|<\frac{\Psi(Q)}{|Q|}\big\},
\end{equation} and
put for any $d\in\mathbb{N}$,
\begin{equation}
W^{(d)}(\Psi)=\bigcap_{n=1}^{\infty}\bigcup_{\partial Q\geq
n}{\mathcal E}_Q(\Psi)^d.
\end{equation}
A few years ago Inoue and Nakada (\cite{Inoue,InoueNakada}) first
studied the metric simultaneous Diophantine approximation in
$\mathbb{F}((X^{-1}))^d$ by attacking the following Duffin-Schaeffer
type conjecture (\cite[Conjecture]{Inoue}):

\begin{Conjecture}\label{conj76}
$\nu_d(W^{(d)}(\Psi))=1$ if and only if
\begin{equation}\sum_{Q\ \mbox{is monic}}\nu_d({\mathcal E}_Q(\Psi)^{d})=\infty.\end{equation}
\end{Conjecture}

 Similar to
Conjectures \ref{Conjecture 16} \& \ref{conj711}  we propose

\begin{Conjecture}\label{conj77}
There exists a universal constant $C>0$ depending only on $d$
and the size of $\mathbb{F}$ such that for any non-negative function
$\Psi:\mathbb{F}[X]\rightarrow\mathbb{R}$,
\begin{equation}
\nu_d\big(\bigcup_{Q\ \mbox{is monic}}{\mathcal
E}_Q(\Psi)^d\big)\geq C\min\{\sum_{Q\ \mbox{is
monic}}\big(\frac{\Psi(Q)\Phi(Q)}{|Q|}\big)^d,1\}.
\end{equation}
\end{Conjecture}

Similar to Theorems \ref{theorem 17} \& \ref{theorem 712} we can
show
\begin{theorem}\label{theorem 78}
Conjecture \ref{conj76} and Conjecture \ref{conj77} are equivalent.
\end{theorem}

The proof of Theorem \ref{theorem 78} is fully identical to that of
Theorem \ref{theorem 712}.

In contrast to the classical and $p$-adic cases (\cite{Haynes
p-adic,PV2}), progresses on Conjecture \ref{conj76} are rather
incomplete. For example,  the  Sprind\u{z}uk  type conjecture
(\cite{Haynes p-adic,PV2,Sprindzuk}) over formal Laurent series
hasn't been established yet. Maybe the best currently known result
is due to Inoue, Nakada (\cite[Thm. 1]{InoueNakada}) and Fuchs
(\cite[Thm. 1]{Fuchs}) whose theorems confirmed Conjecture
\ref{conj76} under the additional assumption that $\Psi(Q)$ depends
only on the degree of $Q$.

Without any extra assumptions on $\Psi$, we will also study a
variant of $W^{(d)}(\Psi)$ by establishing a Gallagher  type theorem
(\cite[Thm. 1]{GallagherHigher}). For any non-negative function
$\Psi:\mathbb{F}[X]\rightarrow\mathbb{R}$, any monic
$Q\in\mathbb{F}[X]$ and any $d\in\mathbb{N}$, we first define
\begin{equation} {\mathcal
H}_Q^{(d)}(\Psi)=\mathbb{L}^d\cap\Big(\bigcup_{{P_i\in\mathbb{F}[X]\atop
\partial P_i<\partial Q}\atop
(P_1,P_2,\cdots,P_d,Q)=1}\prod_{i=1}^dB(\frac{P_i}{Q},\frac{\Psi(Q)}{|Q|})
\Big)
\end{equation}
then set
\begin{equation}
{\mathcal H}^{(d)}(\Psi)=\bigcap_{n=1}^{\infty}\bigcup_{\partial
Q\geq n}{\mathcal H}_Q^{(d)}(\Psi).
\end{equation}
Note ${\mathcal H}^{(1)}(\Psi)=W^{(1)}(\Psi)$. In the
higher-dimensions we will show
\begin{theorem}\label{theorem 114new}
Let $d\geq2$. Then $\nu_d({\mathcal H}^{(d)}(\Psi))=1$ if and only
if
\begin{align}
\sum_{Q\ \mbox{is monic}}\nu_d({\mathcal
H}_Q^{(d)}(\Psi))=\infty.
\end{align}
\end{theorem}

\subsection{Quasi-independence on average method}

In this paper a weighted version of the second Borel-Cantelli lemma
will be introduced for the first time to the study of the
Duffin-Schaeffer conjecture. We begin with a beautiful result of
Gallagher (\cite{Gallagher}) called ``zero-one law" claiming that
$\lambda(W(\psi))$ can be either 0 or 1 for any non-negative
function $\psi$. This means if $\lambda(W(\psi))>0$, then we must
have $\lambda(W(\psi))=1$. A useful tool for proving
$\lambda(W(\psi))>0$ is the following second Borel-Cantelli lemma
due to Erd\"{o}s and R\'{e}nyi (\cite{ErdosRenyi}):

\begin{lemma}\label{lemma 16}
Let $\{{\mathcal A}_n\}_{n\in\mathbb{N}}$ be a sequence of events in
a probability space $(\Omega,\mathbb{P})$ such that
$\sum_n\mathbb{P}({\mathcal A}_n)=\infty$. Then
\begin{equation}\label{formula 13}
\mathbb{P}(\limsup_{n\rightarrow\infty}{\mathcal
A}_n)\geq\limsup_{n\rightarrow\infty}\frac{\displaystyle\big(\sum_{k=1}^n\mathbb{P}({\mathcal
A}_k)\big)^2}{\displaystyle\sum_{i=1}^n\sum_{j=1}^n\mathbb{P}({\mathcal
A}_i\cap{\mathcal A}_j)}.
\end{equation}
\end{lemma}

To the author's knowledge, the proofs of all known results towards
the Duffin-Schaeffer conjecture regard Lemma \ref{lemma 16} as an
indispensable tool (see \cite{Haynes p-adic}), showing under various
additional conditions the quasi-independence on average property for
$\{{\mathcal E}_n(\psi)\}$, that is, proving
\begin{equation}\label{formula 19}
\sum_{m=1}^N\sum_{n=1}^N\lambda({\mathcal E}_m\cap{\mathcal
E}_n)\ll\big(\sum_{n=1}^N\lambda({\mathcal E}_n)\big)^2
\end{equation}
for infinitely many $N\in\mathbb{N}$.
 Recently, Feng, Shen and the author
(\cite{FengLiShen}) generalized the above lemma to

\begin{lemma}\label{lemma 17}Let $\{{\mathcal A}_n\}_{n\in\mathbb{N}}$ be a sequence of events in
a probability space $(\Omega,\mathbb{P})$ and let
$\{\omega_n\}_{n\in\mathbb{N}}$ be a sequence of real numbers such
that $\sum_n\omega_n\mathbb{P}({\mathcal A}_n)=\infty$. Then
\begin{equation}\label{formula 14}
\lambda(\limsup_{n\rightarrow\infty}{\mathcal
A}_n)\geq\limsup_{n\rightarrow\infty}\frac{\displaystyle\big(\sum_{k=1}^n\omega_k\mathbb{P}({\mathcal
A}_k)\big)^2}{\displaystyle\sum_{i=1}^n\sum_{j=1}^n\omega_i\omega_j\mathbb{P}({\mathcal
A}_i\cap{\mathcal A}_j)}.\end{equation}
\end{lemma}


The novelty of Lemma \ref{lemma 17} in our study of the
Duffin-Schaeffer conjecture is that we can choose $\omega_n$ to be
some particular fraction numbers to obtain some good effects. For
example we can show
\begin{theorem}\label{theorem 19} Let $\psi:\mathbb{N}\rightarrow\mathbb{R}$ be a
non-negative function.  Then $\lambda(W(\psi))=1$ if
\begin{equation}\label{formula 111111111111}
\sum_{h: S_h\geq 3}\frac{\log S_h}{h\cdot \log\log S_h}=\infty,
\end{equation}
 where
\begin{equation}\label{formula 17}
S_h=\sum_{n=2^{2^h}+1}^{2^{2^{h+1}}}\frac{\psi(n)\varphi(n)}{n}.
\end{equation}
\end{theorem}

This generalizes a  recent result by Beresnevich et al. (\cite[Thm.
2]{BHHV}) who showed $\lambda(W(\psi))=1$ if there exists a constant
$c>0$ such that
\begin{equation}
\sum_{n=16}^{\infty}\frac{\varphi(n)\psi(n)}{n\exp(c(\log\log
n)(\log\log\log n))}=\infty.
\end{equation}

In metric number theory  the second Borel-Cantelli lemma is so
frequently used in the study of many other problems (\cite{Harman
2000}), for which one may naturally expect that Lemma \ref{lemma 17}
could bring new insight as well as new results.

This paper is mainly arranged as follows:
\begin{itemize}
\item Sections \ref{dsc1}$\sim$\ref{dsc3} are devoted to establishing various
equivalent forms for the classical Duffin-Schaeffer conjecture. In
particular, a general principle will be introduced after the proof
of Theorem \ref{theorem 17}.

\item Sections \ref{padic}, \ref{Laurent} are devoted to studying  the
Duffin-Schaeffer type conjectures in the fields of $p$-adic numbers
and formal Laurent series, respectively.

\item Section
\ref{App BC} is devoted to the proof of Theorem \ref{theorem 19}.
\end{itemize}

\section{Preliminaries}

\subsection{Hausdorff measure}

A dimension function $f:\mathbb{R}^{+}\rightarrow\mathbb{R}^{+}$ is
a continuous, non-decreasing function such that $f(r)\rightarrow0$
as $r\rightarrow0$. A $\rho$-cover of a subset $A$ of $\mathbb{R}$
is a countable collection $\{B_i\}$ of open intervals in
$\mathbb{R}$ with radii $r_i\leq\rho$ for each $i$ such that
$A\subset\bigcup_iB_i$. Define
\[{\mathcal H}_{\rho}^f(A)=\inf\sum_if(r_i),\]
where the infimum is taken over all $\rho$-covers of $A$. The
Hausdorff $f$-measure of $A$ is defined as
\[{\mathcal H}_{\rho}(A)=\lim_{\rho\rightarrow0}{\mathcal H}_{\rho}^f(A).\]
For detailed discussions of the Hausdorff measure theory,  we refer
the readers to the classical book \cite{Falconer} by Falconer.

\subsection{Series} It is well-known  that there is no fastest
converging or slowest diverging series, that is, for any convergent
non-negative series $\sum_{n}a_n$, there exists a sequence of
(increasing) real numbers $\{x_n\}$ with $x_n\rightarrow\infty$ as
$n\rightarrow\infty$, such that $\sum_{n}a_nx_n$ is also convergent;
 and for any
divergent non-negative series $\sum_{n}b_n$, there exists a sequence
of (decreasing) real numbers $\{y_n\}$ with $y_n\rightarrow0$ as
$n\rightarrow\infty$, such that $\sum_{n}b_ny_n$ is also divergent.
In fact, one can choose (\cite[Exer. 11 \& 12, Chap. 3]{Rudin})
\[
x_n=\frac{1}{\displaystyle\sqrt{\sum_{k=n}^{\infty}(a_k+\frac{1}{k^2})}},\
\  y_n=\frac{1}{\displaystyle\sum_{k=1}^n(b_k+\frac{1}{k^2})}.
\]
We may further require that $x_{n+1}-x_n\leq1$ for all $n$ as if not
then we can replace $\{x_n\}$ with $\{z_n\}$ defined recursively by
$z_1=x_1$,
\[z_{n+1}=z_n+\min\{1,x_{n+1}-x_n\}\ \ \ (n\in\mathbb{N}).\]
The reason is very simple and left as an exercise to the interested
readers.

\subsection{Upper bounds}

 In the following we prepare some upper bounds for
$\lambda({\mathcal E}_m\cap {\mathcal E}_n)$, and always assume
$m\neq n$, $m,n\geq2$. First, Duffin and Schaeffer (\cite[Lemma
II]{DuffinSchaeffer}, see also \cite[formula (5)]{BHHV}) proved that
\begin{equation}\label{duffin schaeffer estimate}\lambda({\mathcal E}_m\cap {\mathcal
E}_n)\leq8\psi(m)\psi(n).\end{equation} Second, we have
\begin{equation}\label{formula 2222}\lambda({\mathcal E}_m\cap {\mathcal
E}_n)\ll \lambda({\mathcal E}_m) \lambda({\mathcal
E}_n)P(m,n),\end{equation} where \[P(m,n)=\prod_{p|B(m,n),
p>D(m,n)}(1-\frac{1}{p})^{-1},\] with
$B(m,n)\triangleq\frac{mn}{(m,n)^2}$,
$D(m,n)\triangleq\frac{1}{(m,n)}\cdot\max\{n\psi(m),m\psi(n)\}.$ If
$\{p:p|B(m,n), p>D(m,n)\}=\emptyset$, we understand that
$P(m,n)\equiv1$. This estimate was first stated by Strauch
(\cite{Strauch}), but was also given independently by Pollington and
Vaughan (\cite{PV2}).
 By
one of Merten's theorems (\cite[Theorem 328]{Hardy}) we see that if
$P(m,n)>1$ and $D(m,n)\geq\frac{1}{2}$ then
\begin{equation}\label{formula 25}
P(m,n)\ll\exp\Big(\sum_{D(m,n)<p<\log
B(m,n)}\frac{1}{p}\Big)\ll\frac{\log\log B(m,n)}{2+\log D(m,n)}.
\end{equation}
Finally, we can find in \cite{PV2} that if $D(m,n)<0.5$, then
${\mathcal E}_m\cap {\mathcal E}_n$ is an empty set. This fact is
also easily implied by formula (10) in \cite{BHHV}.

Formula (\ref{formula 25}) is very useful and we will explain in
more detail. Throughout the paper for any $h\in\mathbb{N}$ and any
non-negative function $\psi$, we denote
\begin{itemize}
\item $\Delta_h=\mathbb{N}\cap[2^{2^h}+1,2^{2^{h+1}}]$,
\item
$S_h=S_h(\psi)=\sum_{n\in\Delta_h}\frac{\psi(n)\varphi(n)}{n}$,
\item $B_h=B_h(\psi)=\lambda(\bigcup_{n\in\Delta_h}{\mathcal E}_n)$,
\item $Q_h=Q_h(\psi)={\sum_{(m,n)\in\Delta_h\times\Delta_h,m\neq
n}\lambda({\mathcal E}_m\cap{\mathcal E}_n)}$,
\item $R_h=R_h(\psi)={Q_h}/{S_h^2}$.
\end{itemize}
 To attack the
Duffin-Schaeffer conjecture we need first assume $\sum_hS_h=\infty$.
Note there exists an integer $i\in\{0,1,2\}$ such that
$\sum_hS_{3h+i}=\infty$. By appealing to the Erd\"{o}s-Vaaler
theorem (\cite{Erdos,Vaaler}), to prove $\lambda(W(\psi))=1$ we may
assume without loss of generality that $\psi(n)\geq\frac{1}{n}$
whenever $\psi(n)\neq0$.  Now for any two distinct positive integers
$h_1<h_2$ and for any $m\in\Delta_{3h_1+i}$, $n\in\Delta_{3h_2+i}$,
we have $\log\log B(m,n)\ll h_2$ and
$D(m,n)\geq\frac{n}{m^2}\geq\sqrt{n}$, which in turn gives
\begin{align}\label{formula 2999}
P(m,n)\ll\max\{\frac{\log\log B(m,n)}{2+\log
D(m,n)},1\}\ll\max\{\frac{h_2}{2^{h_2}},1\}\ll1. \end{align} We
should remark that the above kind of arguments was first observed by
Haynes, Pollington and Velani (\cite{Haynes}, see also
\cite{Aistleitner,BHHV}). Hence to study the Duffin-Schaeffer
conjecture it brings no harm for us to assume $\sum_hS_h=\infty$
together with:
\begin{align}\label{formula 25555}
P(m,n)\ll1\ \mbox{for any}\ m,n\ \mbox{in any corresponding distinct
blocks}\ \Delta_{h_1},\Delta_{h_2}.\end{align}

\section{Equivalence  of the Duffin-Schaeffer conjecture
(1)}\label{dsc1}

The section is mainly devoted to the proof of Theorem \ref{theorem
17}. A  general principle and a higher-dimensional analogue will
also be established.

\subsection{Proof of Theorem \ref{theorem 17}}

For any non-negative function
$\psi:\mathbb{N}\rightarrow\mathbb{R}$, we denote
\begin{equation}S(\psi)=\sum_{n=1}^{\infty}\frac{\psi(n)\varphi(n)}{n}.\end{equation}
 For
any $N\in\mathbb{N}$,  we set
\begin{equation}A_N=\inf\{\lambda(Z(\psi)): \mbox{supp}(\psi)\subset[N,\infty)\ \ \mbox{is bounded}, S(\psi)\geq1\}.\end{equation}
Obviously, $0\leq A_1\leq A_2\leq A_3\leq\cdots\leq 1$. So we can
define
\begin{equation}
A_{\infty}=\lim_{N\rightarrow\infty}A_N.
\end{equation}

\textsc{Claim 1}: The Duffin-Schaeffer conjecture is true if and
only if $A_{\infty}>0$.

\textsc{Proof of Claim 1}: ``$\Leftarrow$" Suppose $A_{\infty}>0$.
Let $\psi$ be any non-negative function such that $S(\psi)=\infty$.
We need to show that $\lambda(W(\psi))=1$. Let $N_1<N_2<N_3<\cdots$
be any fixed sequence of positive integers such that
\[\sum_{n=N_k+1}^{N_{k+1}}\frac{\psi(n)\varphi(n)}{n}\geq1.\]
 By definition if $k$ is large enough,
then
\[\lambda(\bigcup_{n=N_k+1}^{N_{k+1}}{\mathcal E}_n)\geq\frac{A_{\infty}}{2}.\]
By the continuity of the Lebesgue measure we have
\[\lambda(W(\psi))=\lim_{N\rightarrow\infty}\lambda\big(\bigcup_{n=N}^{\infty}{\mathcal E}_n\big)\geq\frac{A_{\infty}}{2},\]
which gives $\lambda(W(\psi))=1$ by applying Gallagher's zero-one
law. This proves the Duffin-Schaeffer conjecture under the
assumption $A_{\infty}>0$.

``$\Rightarrow$" Suppose the Duffin-Schaeffer conjecture is true. We
argue by contradiction and suppose $A_{\infty}=0$. Thus $A_N=0$ for
any $N\in\mathbb{N}$. Obviously, be the definitions of $A_N$ we can
find a sequence of non-negative functions $\{\psi_k\}$ and a
 sequence of positive integers $N_1<N_2<N_3<\cdots$  such that
 $\mbox{supp}(\psi_k)\subset[N_k+1,N_{k+1}]$,
\[\sum_{n=N_k+1}^{N_{k+1}}\frac{\psi_k(n)\varphi(n)}{n}\geq1,\]
and \[\lambda(\bigcup_{n=N_k+1}^{N_{k+1}}{\mathcal
E}_n(\psi_k))\leq\frac{1}{k^2}.\] Gluing this sequence of disjointly
supported functions $\{\psi_k\}$ into a new function $\psi$, it is
easy to deduce from the first Borel-Cantelli lemma that
$\lambda(W(\psi))=0$. Note $S(\psi)=\infty$. So we get a
contradiction to the assumed truth of Duffin-Schaeffer conjecture.
This finishes the whole proof of Claim 1.

\textsc{Claim 2}: $A_{1}>0\Leftrightarrow A_{\infty}>0$.

\textsc{Proof of Claim 2}: Obviously, $A_{1}>0$ implies
$A_{\infty}>0$. So we need only to show that $A_{\infty}>0$ implies
$A_{1}>0$. By the continuity of the Lebesgue measure it suffices to
give a universal lower bound for $\lambda(Z(\psi))$, where $\psi$ is
any non-negative function with bounded support and $S(\psi)\geq1$.
To this aim we first choose an $N\in\mathbb{N}$ such that
$A_N\geq\frac{A_{\infty}}{2}$, then decompose $\psi$ as the sum of
two unique non-negative functions $\psi_1$, $\psi_2$ with
corresponding bounded supports in $[1,N-1]$ and $[N,\infty)$. Now we
have two cases to consider.

Case 1: Suppose $S(\psi_2)\geq1-\frac{A_{\infty}}{8}$. Let $\psi_3$
be any non-negative function with bounded support in
$(\max\mbox{supp}(\psi_2),\infty)$ such that
$S(\psi_3)=\frac{A_{\infty}}{8}$. Note
\[\frac{A_{\infty}}{2}\leq\lambda(Z(\psi_2+\psi_3))\leq\lambda(Z(\psi_2))+2S(\psi_3)\leq
\lambda(Z(\psi_2))+\frac{A_{\infty}}{4},\] from which we deduce
$\lambda(Z(\psi))\geq\lambda(Z(\psi_2))\geq\frac{A_{\infty}}{4}.$

Case 2: Suppose $S(\psi_2)<1-\frac{A_{\infty}}{8}$. Choose an $n<N$
such that
$\frac{\psi(n)\varphi(n)}{n}\geq\frac{A_{\infty}}{8(N-1)}.$ Applying
the left hand side of (\ref{formula 141414})  gives
\[\lambda(Z(\psi))\geq\lambda({\mathcal E}_n)\geq\min\{\frac{A_{\infty}}{8(N-1)},\frac{1}{2}\}.\]
This finishes the whole proof of Claim 2.

\textsc{Claim 3}: For any $t\geq1$ and any non-negative function
$\psi$,  $\lambda(Z(t\psi))\leq t\lambda(Z(\psi))$.

\textsc{Proof of Claim 3}: By the continuity of the Lebesgue
measure, we may assume without loss of generality that $\psi$ is of
bounded support, and suppose this is the case. Since ${\mathcal
E}_n$ is an open set in $\mathbb{R}/\mathbb{Z}$ for any
$n\in\mathbb{N}$, we see that $Z(\psi)$ is the union of finitely
many pairwise disjointly supported open subsets
 of  $\mathbb{R}/\mathbb{Z}$, say for example,
\begin{align}\label{claim 3 proof}Z(\psi)=\bigcup_{k=1}^M\{e^{2\pi iy}:y\in
 I_k\},\end{align}
 where
 $I_k=(x_k-r_k,x_k+r_k)$, $k=1,2,\ldots,M$, are pairwise disjointly supported open intervals in
 $(-1,2)$.
 For the sake of simplicity we identify   $\{e^{2\pi iy}:y\in
 I_k\}$ with $I_k$.
 Now suppose $x\in Z(t\psi)$. This means  one can find a coprime pair
 $(m,n)$ such that
$|x-\frac{m}{n}|<\frac{t\psi(n)}{n}.$
 According to the decomposition (\ref{claim 3 proof}), there exists
 an $I_k$ such that
 \[(\frac{m}{n}-\frac{\psi(n)}{n},\frac{m}{n}+\frac{\psi(n)}{n})\subset I_k.\]
Comparing the lengths of the above two intervals, we get
$\frac{\psi(n)}{n}\leq r_k$. Consequently,
\begin{align*}x&\in(\frac{m}{n}-\frac{t\psi(n)}{n},\frac{m}{n}+\frac{t\psi(n)}{n})\\&=(\frac{m}{n}-\frac{\psi(n)}{n}-\frac{(t-1)\psi(n)}{n},
\frac{m}{n}+\frac{\psi(n)}{n}+\frac{(t-1)\psi(n)}{n})\\
&\subset(x_k-r_k-(t-1)r_k,x_k+r_k+(t-1)r_k)\\
&=(x_k-tr_k,x_k+tr_k),\end{align*} which naturally implies that
\[Z(t\psi)\subset\bigcup_{k=1}^M(x_k-tr_k,x_k+tr_k).\]
So we have
$\lambda(Z(t\psi))\leq\sum_{k=1}^M2tr_k=t\lambda(Z(\psi)).$ This
finishes the proof of Claim 3.

\textsc{Proof of Theorem \ref{theorem 17}}: We first note that if
Conjecture \ref{Conjecture 16} is true, then $A_1>0$, or
equivalently by Claims 1 \& 2, the Duffin-Schaeffer conjecture is
true. Next, we assume the Duffin-Schaeffer conjecture is true and
are going to show that Conjecture \ref{Conjecture 16} is  also true.
To this aim it suffices to establish for  any non-negative function
$\psi$ that
\[\lambda(Z(\psi))\geq A_1\min\{S(\psi),1\},\]
where $A_1>0$ follows from Claims 1 \& 2. By the continuity of the
Lebesgue measure, we may further assume without loss of generality
that $\psi$ is of non-empty bounded support. Now we have two cases
to consider.

Case 1: Suppose $S(\psi)\geq1$. By the definition of $A_1$ we have
$\lambda(Z(\psi))\geq A_1$.

Case 2: Suppose $S(\psi)<1$. Let
\[t\triangleq \frac{1}{S(\psi)}>1,\]
which means $S(t\psi)=1$. By  the definition of $A_1$,
$\lambda(Z(t\psi))\geq A_1$. By Claim 3,
\[\lambda(Z(\psi))\geq\frac{\lambda(Z(t\psi))}{t}\geq A_1S(\psi).\] This
finishes the whole proof of Theorem \ref{theorem 17}.

\subsection{A general principle} Let $\{\Omega,{\mathcal
F},\mathbb{P}\}$ be a probability space and let ${\mathcal F}_n$
($n\in\mathbb{N}$) be a fixed subset of ${\mathcal F}$. $\{{\mathcal
F}_n\}_n$ is said to have the Duffin-Schaeffer property if
\begin{equation}\sum_{n=1}^{\infty}\mathbb{P}(E_n)=\infty\Rightarrow
\mathbb{P}(\limsup_{n\rightarrow\infty}E_n)=1
\end{equation} for any sequence
of events $\{E_n\}_n$ with $E_n\in{\mathcal F}_n$
($n\in\mathbb{N}$). Similarly, $\{{\mathcal F}_n\}_n$ is said to
have the zero-one property if
\begin{equation}
\mathbb{P}(\limsup_{n\rightarrow\infty}E_n)\in\{0,1\}
\end{equation}
for any sequence of events $\{E_n\}_n$ with $E_n\in{\mathcal F}_n$
($n\in\mathbb{N}$). A general principle implied in the proofs of
Claims 1 \& 2
 is the following theorem whose proof is left as a simple exercise to the
 interested readers.

\begin{theorem}\label{theorem 71}
Suppose $\{{\mathcal F}_n\}_n$ is of the zero-one property. Then
$\{{\mathcal F}_n\}_n$ is of the Duffin-Schaeffer property if and
only if
\begin{equation}\inf\{\mathbb{P}(\bigcup_{n=1}^{\infty}E_n):\sum_{n=1}^{\infty}\mathbb{P}(E_n)\geq1, E_n\in{\mathcal F}_n\}>0.
\end{equation}
\end{theorem}

Note in the simultaneous and multiplicative Diophantine
approximation the zero-one property is in general not a big problem
(see e.g. \cite{BHV multi,BV08,Li1}) as we have the cross fibering
principle due to Beresnevich, Haynes and Velani (\cite{BHV multi}).
Many conjectures in metric number theory were formulated to prove a
particularly chosen $\{{\mathcal F}_n\}_n$ having the
Duffin-Schaeffer property (see e.g. \cite{BBDV,BHV multi,Li1}).
Hence Theorem \ref{theorem 71} provides a new way to look at such
kind of conjectures.

We remark that unconditionally one cannot expect
\begin{equation}\label{formula 733}
\mathbb{P}(\bigcup_{n=1}^{\infty}E_n)\asymp\min\{\sum_{n=1}^{\infty}\mathbb{P}(E_n),1\}
\end{equation}
as we have the following example in $\mathbb{R}/\mathbb{Z}$: define
\begin{align*}
{\mathcal F}_n&=\begin{cases}
\{\emptyset,[0,1)\} &  (n\ \mbox{is odd}) \\
\{\emptyset,[0,\frac{1}{n^2})\} & (n\ \mbox{is even}),
\end{cases}\\
E_n&=\begin{cases}
\emptyset &  (n\ \mbox{is odd or}\ n\ \mbox{is even with}\ n\leq 2N) \\
[0,\frac{1}{n^2}) & (n\ \mbox{is even with}\ n>2N),
\end{cases}
\end{align*}
where $N$ is large enough. Even though, one may still expect the
equivalence between (\ref{formula 733}) and the corresponding
Duffin-Schaeffer-type conjecture for some particularly chosen
$\{{\mathcal F}_n\}_n$ (see e.g. Theorems \ref{theorem 17},
\ref{theorem 78}, \ref{theorem 712} \& \ref{theorem 72}).

\subsection{Higher-dimensional case}

We are going to generalize Theorem \ref{theorem 17} to the higher
dimensions.

\begin{theorem}\label{theorem 72} Let $d\geq 2$.
Then
\begin{equation}\lambda_d\big(\bigcup_{n=1}^{\infty}{\mathcal E}_n^{d}\big)\geq
C_d\min\{\sum_{n=1}^{\infty}(\frac{\psi(n)\varphi(n)}{n})^d,1\},\end{equation}
where  $\lambda_d$ denotes the $d$-dimensional Lebesgue measure on
$(\mathbb{R}/\mathbb{Z})^d$, ${\mathcal E}_n^{d}$ denotes the
$d$-fold product of ${\mathcal E}_n$, $C_d>0$ depends only on $d$.
\end{theorem}

The key to the proof of Theorem \ref{theorem 72} is the following
generalization of Claim 3 in the higher-dimensions as Claims 1 \& 2
have just been generalized by Theorem \ref{theorem 71}, and the
Sprind\u{z}uk conjecture (\cite{Sprindzuk}, also known as the
higher-dimensional Duffin-Schaeffer conjecture) was confirmed by
Pollington and Vaughan (\cite{PV2}, see also \cite{Harman}).

\begin{lemma}\label{lemma 73} Let $d\geq 1$ and $t\geq1$.
Then there exists a constant $C_d>0$ depending only on $d$ such that
\begin{equation}\lambda_d\big(\bigcup_{n=1}^{\infty}{\mathcal E}_n(t\psi)^d\big)\leq C_dt^d\lambda_d
\big(\bigcup_{n=1}^{\infty}{\mathcal E}_n(\psi)^{d}\big).
\end{equation}
\end{lemma}

\begin{proof}
We only outline the proof in four steps and the interested readers
can easily provide all the details as the proof of Lemma \ref{lemma
73} is similar to that of Claim 3.
 First, by the continuity of the
Lebesgue measure we may assume $\psi$ is of bounded support, thus
$\bigcup_{n=1}^{\infty}{\mathcal E}_n(\psi)^{d}$ is the union of
finitely many  $d$-dimensional cubes $\{C(x_i,r_i)\}$; next, we
apply the classical Vitali covering lemma (see e.g. \cite{Stein}) to
choose a subcollection of disjointly supported cubes
$\{C(y_j,s_j)\}$ such that
\[\bigcup_iC(x_i,r_i)\subset\bigcup_jC(y_j,M_ds_j),\]
where $M_d>0$ is a universal constant depending only on $d$; after
that, via elementary geometric observation it is easy to show that
\[\bigcup_{n=1}^{\infty}{\mathcal E}_n(t\psi)^{d}=\bigcup_iC(x_i,tr_i)\subset\bigcup_jC(y_j,M_d^{(2)}ts_j),\]
where $M_d^{(2)}>0$  depends also only on $d$; finally, comparing
the $d$-dimensional measures of the left and right hand sides of the
last formula gives the desired result.
\end{proof}

\section{Equivalence  of the Duffin-Schaeffer conjecture (2)}\label{dsc2}

This section is devoted to providing more equivalent forms for the
classical Duffin-Schaeffer conjecture. As a byproduct, the question
of Haynes introduced in the first section will be given a
conditional but almost best possible answer.

First we observe from Claims 1 \& 2  in the proof of Theorem
\ref{theorem 17} (see also Theorem \ref{theorem 71}) that the
Duffin-Schaeffer conjecture implies the following

\begin{Conjecture}\label{73}
There exists a universal constant $C>0$ independent of
$h\in\mathbb{N}$ and non-negative function $\psi$ such that if
$S_h(\psi)=1$, then $B_h(\psi)\geq C$.
\end{Conjecture}

By Claim 3 in the proof of Theorem \ref{theorem 17}, this conjecture
is in turn equivalent to

\begin{Conjecture}\label{72}
There exists a universal constant $C>0$ independent of
$h\in\mathbb{N}$ and non-negative function $\psi$ such that if
$S_h(\psi)\leq1$, then $B_h(\psi)\geq CS_h(\psi)$.
\end{Conjecture}

Both conjectures are not weak at all as we have

\begin{theorem}\label{theorem 53} Conjecture \ref{73}, Conjecture \ref{72} and the
Duffin-Schaeffer conjecture are all equivalent.
\end{theorem}

Obviously, to prove this theorem it suffices to show that Conjecture
\ref{72} implies the Duffin-Schaeffer conjecture. To this aim, let
$\psi$ be any non-negative function such that
$\sum_hS_h(\psi)=\infty$.  Without loss of generality we may assume
$S_h(\psi)\leq1$ for all $h\in\mathbb{N}$ as if not then we can
study another function $\psi_2\leq\psi$ defined by
\begin{align*}\psi_2(n)=\begin{cases}
\psi(n) & (n\in\Delta_h, S_h(\psi)\leq1) \\
\frac{\psi(n)}{S_h(\psi)}& (n\in\Delta_h, S_h(\psi)>1)
\end{cases}\end{align*}
to deduce first $\lambda(W(\psi_2))=1$ then $\lambda(W(\psi))=1$.
According to the assumed truth of Conjecture \ref{72} and the
following proposition (see also the proof of \cite[Thm.
1]{Aistleitner}) we have $\lambda(W(\psi))=1$, which proves the
Duffin-Schaeffer conjecture.

\begin{Proposition}\label{prop 74} Let $\psi$ be any non-negative function such that $\sum_hS_h(\psi)=\infty$. Then $\lambda(W(\psi))=1$ if  there
exists a constant $C>0$ independent of $h\in\mathbb{N}$ such that
$B_h(\psi)\geq CS_h(\psi)$ for any $h\in\mathbb{N}$.
\end{Proposition}

\begin{proof}
Thanks to (\ref{formula 25555}) we may  assume without loss of
generality that $P(m,n)\ll1$ for any $m,n$ in any corresponding
distinct blocks $\Delta_{h_1},\Delta_{h_2}$. As $W(\psi)$ is also of
the form
$\limsup_{h\rightarrow\infty}\bigcup_{n\in\Delta_h}{\mathcal
E}_n(\psi)$ and $\sum_h\lambda\big(\bigcup_{n\in\Delta_h}{\mathcal
E}_n(\psi)\big)=\sum_hB_h(\psi)=\infty$, we can apply Lemma
\ref{lemma 16} to deduce
\begin{align*}
\lambda(W(\psi))&\geq\limsup_{N\rightarrow\infty}\frac{\displaystyle\Big(\sum_{h=1}^N\lambda\big(\bigcup_{n\in\Delta_{h}}{\mathcal
E}_n(\psi)\big)\Big)^2}{\displaystyle\sum_{h_1=1}^N\sum_{h_2=1}^N\lambda\Big(\big(\bigcup_{m\in\Delta_{h_1}}{\mathcal
E}_m(\psi)\big)\cap\big(\bigcup_{n\in\Delta_{h_2}}{\mathcal
E}_n(\psi)\big)\Big)}\\
&\geq\limsup_{N\rightarrow\infty}\frac{\displaystyle\Big(\sum_{h=1}^NB_{h}(\psi)\Big)^2}{\displaystyle
\sum_{h=1}^NB_{h}(\psi)+2\sum_{1\leq h_1<h_2\leq
N}\sum_{m\in\Delta_{h_1} \atop
n\in\Delta_{h_2}}\lambda\big({\mathcal E}_m(\psi)\cap{\mathcal
E}_n(\psi)\big)}\\
&\gg\limsup_{N\rightarrow\infty}\frac{\displaystyle\Big(\sum_{h=1}^NS_{h}(\psi)\Big)^2}{\displaystyle
\sum_{h=1}^NS_{h}(\psi)+\Big(\sum_{h=1}^NS_{h}(\psi)\Big)^2}>0,
\end{align*}
which gives $\lambda(W(\psi))=1$ by applying Gallagher's zero-one
law. We are done.
\end{proof}

As introduced in the first section  Haynes asked whether there
exists a non-negative function $\psi$ for which one cannot use the
quasi-independence on average method  to prove $\lambda(W(\psi))=1$.
Based on the equivalence between the classical Duffin-Schaeffer
conjecture and Conjecture \ref{72}, Proposition \ref{prop 74} and
its quasi-independence on average proof, a conditional answer is: If
the Duffin-Schaeffer conjecture is true, then the answer is NO. This
answer is almost best possible as to give an absolutely complete
answer one need first assume that the Duffin-Schaeffer conjecture is
false.

 Haynes (\cite{Haynes p-adic}) also commented that an answer to the
 above question would bring us closer to the heart of the
 Duffin-Schaeffer conjecture. Based on all previous discussions, we believe
 the heart of the  Duffin-Schaeffer conjecture is to establish $B_h\asymp
 S_h$ whenever $S_h\leq1$ for
 \begin{itemize}
 \item If anyone can confirm or disprove $B_h\asymp
 S_h$ when $S_h\leq1$, then the Duffin-Schaeffer conjecture is true or
 false, respectively.
  \item If anyone can confirm $B_h\asymp
 S_h$ when $S_h\leq1$ under additional assumptions, then we can use
 Proposition \ref{prop 74} to obtain some partial results.
 \end{itemize}
For example, any of the following assumptions can give $B_h\asymp
S_h$:
\begin{itemize}
\item (A1)
$\sum\limits_{n\in\Delta_h}\frac{\psi(n)\varphi(n)}{n}\leq\frac{1}{h}$,
\item (A2) $\sum\limits_{n\in\Delta_h}{\psi(n)}\leq\frac{1}{\sqrt{h}}$,
\item (A3) $\sum\limits_{n\in\Delta_h}\frac{\psi(n)n}{\varphi(n)}\leq1$.
\end{itemize}
To this aim we first note from the Cauchy-Schwarz inequality that
\begin{equation}B_h\geq\frac{S_h^2}{\displaystyle S_h+Q_h}.\end{equation}
Thus to get $B_h\asymp S_h$ it suffices to establish $Q_h\ll S_h$.
Since we not only have the powerful (\ref{formula 25}) but also have
the Duffin-Schaeffer estimate (\ref{duffin schaeffer estimate}), it
is rather easy to  deduce $Q_h\ll S_h$ from any  of the above three
assumptions. Just for one example, suppose
$\sum_{n\in\Delta_h}\frac{\psi(n)n}{\varphi(n)}\leq1$. We then have
\begin{equation}
Q_h\ll (\sum\limits_{n\in\Delta_h}{\psi(n)})^2\leq
S_h\cdot\sum\limits_{n\in\Delta_h}\frac{\psi(n)n}{\varphi(n)}\leq
S_h.
\end{equation}
We remark that the assumption (A1) was first observed by Aistleitner
(\cite{Aistleitner}).

Conjectures \ref{73} \& \ref{72} suggest two directions in the study
of the Duffin-Schaeffer conjecture, one  is proving $B_h\asymp1$
under the assumption $S_h\geq f(h)$ for some slowly-increasing
function $f$, the other is showing $B_h\asymp S_h$ under the
assumption $S_h\leq g(h)$ for some slowly-decreasing function $g$.
For example, what Beresnevich et al.  have actually proved on their
\cite[Thm. 2]{BHHV} (one may also deduce it from the proof of
Theorem \ref{theorem 19}) is in principle the following estimate:

\begin{Proposition}\label{45}
There exists a universal constant $C_{\alpha}>0$ depending only on
$\alpha>0$ such that if $S_h\geq \exp({\alpha h\log h})$, then
$B_h\geq C_{\alpha}$.
\end{Proposition}

Finally we explain how has Conjecture \ref{Conjecture 16} been
proposed. Let $\psi$ be a non-negative function such that
$\sum_hS_h(\psi)$ diverges extremely slow to infinity.  If
$\frac{B_h(\psi)}{S_h(\psi)}\rightarrow0$ as $h\rightarrow\infty$,
then it is highly possible that $\sum_hB_h(\psi)<\infty$, which
implies the Duffin-Schaeffer conjecture is false. On the other hand,
if $B_h(\psi)\asymp S_h(\psi)$, then we can use Proposition
\ref{prop 74} to deduce $\lambda(W(\psi))=1$. Motivated by this
observation we believe in general $B_h(\psi)\asymp S_h(\psi)$
whenever $S_h(\psi)\leq1$.

\section{Equivalence  of the Duffin-Schaeffer conjecture (3)}\label{dsc3}

This section is devoted to the proof of Theorem \ref{theorem 15}
which is the combination of two smaller ones.

\begin{theorem}\label{prop 21}
Conjecture \ref{Conjecture 13} is equivalent to the Duffin-Schaeffer
conjecture.
\end{theorem}

Let us briefly explain how will we derive a proof of Theorem
\ref{prop 21}. Obviously, it suffices (\cite{Haynes}) to prove that
Conjecture \ref{Conjecture 13} implies the Duffin-Schaeffer
conjecture. To this aim, we need only to get a contradiction by
assuming the truth of Conjecture \ref{Conjecture 13} and absurdity
of the Duffin-Schaeffer conjecture. Thus suppose there exists a
non-negative function $\psi$ with $S(\psi)=\infty$ such that
$\lambda(W(\psi))=0$. To give a proof of Theorem \ref{prop 21} it
suffices to construct a dimension function $f$ satisfying the
assumption of Conjecture \ref{Conjecture 13} such that ${\mathcal
H}^f(W(\psi))<\infty$. This is indeed possible as we can learn from
the next lemma, hence a proof of Theorem \ref{prop 21} is obtained.

Lemma \ref{lemma 22} might be a standard fact in Hausdorff measure
theory. But the author could not find such a statement from popular
references, so we include a proof here.

\begin{lemma}\label{lemma 22} Let $A\subset \mathbb{R}$ be of zero Lebesgue
measure. Then there exists a dimension function $f$ with
$r^{-1}f(r)\nearrow\infty$ as $r\rightarrow0$, such that ${\mathcal
H}^f(A)=0$.
\end{lemma}

\begin{proof} Since $A$ is of zero Lebesgue measure, for each
$n\in\mathbb{N}$ there exists a sequence of open intervals
$I^{(n)}_i=B(x^{(n)}_i,r^{(n)}_i)$ whose union covers $A$ such that
$\sum_{i}r^{(n)}_i<{1}/{e^n}$.   We partition all the intervals
$\{I^{(n)}_i\}_{n,i}$ into the collections $F_s$ $(s\in\mathbb{N})$
such that the radius of any member of $F_s$ lies in
$[\frac{1}{e^{s+1}},\frac{1}{e^{s}})$. Obviously,
$\sum_{s=1}^{\infty}\frac{|F_s|}{e^s}<\infty.$ As discussed in the
second section, there exists
 an increasing positive function
$g:\mathbb{N}\rightarrow\mathbb{R}$ with $g(n)\nearrow\infty$ as
$n\rightarrow\infty$ and $g(n+1)-g(n)\leq1$ such that
\begin{equation}\label{21}\sum_{n=1}^{\infty}\frac{|F_n|g(n)}{e^n}<\infty.\end{equation}
Obviously, we may further assume $g(1)=1$. The linear interpolation
of $g$ defined on $[1,\infty)$ is denoted still by $g$. Note  for
all non-integer points $x\in[1,\infty)$,
\begin{equation}\label{formula 22}g'(x)\leq1<g(x),\end{equation}
where $g'(x)$ means as usual the derivative of $g$ at $x$. With
these preparations we define a function
$f:(0,\infty)\rightarrow(0,\infty)$ by \[f(r)=r\cdot
g(\max\{1,-1-\log r\}),\] and are going to verify step by step that
the function $f$ satisfies the required claim  as follows:

1) Obviously, $f$ is continuous. It follows from (\ref{formula 22})
that $f$ is increasing and $g(n)\leq n$. Since
\[f(\frac{1}{e^{n+1}})=\frac{g(n)}{e^{n+1}}\leq\frac{n}{e^{n+1}},\]
we see that $f(r)\rightarrow0$ as $r\rightarrow0$. This shows that
$f$ is indeed a dimension function.

2) $r^{-1}f(r)=g(\max\{1,-1-\log r\})\nearrow\infty$ as
$r\rightarrow0$.

3) Note first
\[A\subset\bigcup_{i=1}^{\infty}I^{(n)}_i\subset\bigcup_{k=n}^{\infty}\bigcup_{i=1}^{\infty}I^{(k)}_i,\]
which means $\bigcup_{k=n}^{\infty}\bigcup_{i=1}^{\infty}I^{(k)}_i$
is a $\frac{1}{e^n}$-cover of $A$. Hence if $n\geq2$, then
\[{\mathcal H}^f_{\frac{1}{e^n}}(A)\leq\sum_{k=n}^{\infty}\sum_{i=1}^{\infty}f(r^{(k)}_i)\leq\sum_{s=n}^{\infty}\frac{|F_s|g(s)}{e^s}.\]
In view of (\ref{21}) we must have ${\mathcal H}^f(A)=0$. This
finishes the proof of Lemma \ref{lemma 22}.
\end{proof}

\begin{theorem}\label{prop 23}
Conjecture \ref{Conjecture 14} is equivalent to the Duffin-Schaeffer
conjecture.
\end{theorem}

\begin{proof} Obviously, it suffices (\cite{Haynes}) to prove that
Conjecture \ref{Conjecture 14} implies the Duffin-Schaeffer
conjecture. Suppose Conjecture \ref{Conjecture 14} is  true, and let
$\psi:\mathbb{N}\rightarrow\mathbb{R}$ be any non-negative function
such that $\sum_n\frac{\varphi(n)\psi(n)}{n}=\infty$. Our purpose
below is to show that $\lambda(W(\psi))=1$. By appealing to the
Erd\"{o}s-Vaaler theorem (\cite{Erdos,Vaaler}) and to a theorem of
Pollington and Vaughan (\cite[Thm. 2]{PV2}), we can assume without
loss of generality that $1/n\leq \psi(n)\leq1/2$ whenever
$\psi(n)\neq0$. As discussed in the second section, there exists a
decreasing positive function $g:\mathbb{N}\rightarrow\mathbb{R}$
with $g(n)\searrow0$ as $n\rightarrow\infty$ such that
\[\sum_n\frac{\varphi(n)\psi(n)g(n)}{n}=\infty.\] The linear
interpolation of $g$ defined on $[1,\infty)$ is denoted still by
$g$. Considering $g(n)\searrow0$ as $n\rightarrow\infty$, it is
convenient for us to define $g(\infty)=0$.
  With these preparations
we now define a non-negative function
$f:[0,\infty)\rightarrow\mathbb{R}$ by
$$f(r)=r\cdot g\big(\max\{\frac{1}{\sqrt{r}},1\}\big),$$
 and are going to verify step by step that
the function $f$ satisfies the required claim  as follows:

1) $f$ is an increasing function as it is the product of the
increasing function $r\mapsto r$ and the non-decreasing function
$r\mapsto g(\max\{\frac{1}{\sqrt{r}},1\})$.

2) $r^{-1}f(r)=g(\max\{\frac{1}{\sqrt{r}},1\})\rightarrow0$ as
$r\rightarrow0$.

3) Considering $g(\infty)=0$ and $1/n\leq \psi(n)\leq1/2$ whenever
$\psi(n)\neq0$, we have
\[\sum_{n\in\mathbb{N}}f(\frac{\psi(n)}{n})\varphi(n)=\sum_{n\in\mathbb{N}}\frac{\psi(n)}{n}\varphi(n)g(\sqrt{\frac{n}{\psi(n)}})
\geq \sum_{n\in\mathbb{N}}\frac{\psi(n)}{n}\varphi(n)g(n)=\infty.
\] Hence $\lambda(W(\psi))=1$ follows from the assumed
truth of Conjecture \ref{Conjecture 14}. This finishes the proof of
Theorem \ref{prop 23}.
\end{proof}

\section{$p$-adic approximation}\label{padic}

This section is devoted to the proof of Theorem \ref{theorem 712}.
We will also discuss Haynes' question in the new setting of $p$-adic
numbers.

\subsection{Proof of Theorem \ref{theorem 712}}
 Every non-zero $p$-adic
number $\alpha$ has a unique $p$-adic expansion
\[\alpha=\sum_{k=s}^{\infty}\alpha_kp^k\ \ \  (\alpha_k\in\mathbb{Z},\ 0\leq \alpha_k\leq p-1,\ \alpha_s>0),\]
and we can do arithmetic in $\mathbb{Q}_p$ in similar fashion to the
way it is done in $\mathbb{R}$ with decimal expansions. With this
form $|\alpha|_p\triangleq p^{-s}$, and $\alpha\in\mathbb{Z}_p$ if
and only if $\alpha_k=0$ whenever $k<0$. By the
translation-invariant property of the Haar measure $\mu_p$ on
$\mathbb{Q}_p$ and by $\mu_p(\mathbb{Z}_p)=1$, we see that for any
$\beta\in\mathbb{Q}_p$ and any $z\in\mathbb{Z}$,
$\mu_p(\overline{B}(\beta,p^z))=p^z$, which implies further for any
$r>0$, $\mu_p(\overline{B}(\beta,r))\leq r$ ($\spadesuit$).
  With these preparations we can
generalize Lemma \ref{lemma 73}  to the $p$-adic case.

\begin{lemma}\label{lemma 71padic} Let $t\geq1$. Then  for any non-negative function
$\psi:\mathbb{N}\rightarrow\mathbb{R}$,
\[\mu_p\big(\bigcup_{n=1}^{\infty}{\mathcal K}_n(t\psi)\big)\leq t\mu_p\big(\bigcup_{n=1}^{\infty}{\mathcal K}_n(\psi)\big).\]
\end{lemma}

\begin{proof}[Proof of Lemma \ref{lemma 71padic}]
The proof of Lemma \ref{lemma 71padic} is similar to that of Lemma
\ref{lemma 73}, and in fact is more simpler. By the inner regular
property of the Haar measure $\mu_p$ ($\bigstar$), we may assume
without loss of generality that $\psi$ is  non-empty bounded
support. As any two closed balls in $\mathbb{Q}_p$ can only have
either empty intersection or one is contained in the other
($\clubsuit$), we see that $\bigcup_{n}{\mathcal K}_n(\psi)$ is the
union of finitely many pairwise disjointly supported non-empty
closed balls of the following form
\[\bigcup_{n}{\mathcal K}_n(\psi)=\bigcup_{i=1}^M\bigcup_{j=1}^{M_i}\big(\overline{B}(\frac{a_{i,j}}{n_i},\frac{\psi(n_i)}{n_i})\cap\mathbb{Z}_p\big)\ \
\ \ -n_i\leq a_{i,j}\leq n_i,\ (a_{i,j},n_i)=1,\] and
\[\bigcup_{n}{\mathcal K}_n(t\psi)=\bigcup_{i=1}^M\bigcup_{j=1}^{M_i}\big(\overline{B}(\frac{a_{i,j}}{n_i},\frac{t\psi(n_i)}{n_i})\cap\mathbb{Z}_p\big).\]
If there exists an pair $(i,j)$ such that $\mathbb{Z}_p\subset
\overline{B}(\frac{a_{i,j}}{n_i},\frac{\psi(n_i)}{n_i})$,  then
$\bigcup_{n}{\mathcal K}_n(\psi)=\mathbb{Z}_p$ and we need to do
nothing further. Else by the property $\clubsuit$ we can assume
$\overline{B}(\frac{a_{i,j}}{n_i},\frac{\psi(n_i)}{n_i})\subset\mathbb{Z}_p$
for all the pairs $(i,j)$. Consequently, by the property
$\spadesuit$ we
have\begin{align*}\mu_p\Big(\bigcup_{n=1}^{\infty}{\mathcal
K}_n(t\psi)\Big)&\leq
\sum_{i=1}^M\sum_{j=1}^{M_i}\mu_p\Big(\overline{B}(\frac{a_{i,j}}{n_i},\frac{t\psi(n_i)}{n_i})\Big)\\
&\leq
t\sum_{i=1}^M\sum_{j=1}^{M_i}\mu_p\Big(\overline{B}(\frac{a_{i,j}}{n_i},\frac{\psi(n_i)}{n_i})\Big)=t\mu_p\Big(\bigcup_{n=1}^{\infty}{\mathcal
K}_n(\psi)\Big).\end{align*} This finishes the proof of Lemma
\ref{lemma 71padic}.
\end{proof}

In the $p$-adic case we also have the Gallagher type zero-one law
(\cite[Lemma 1]{Haynes p-adic}). Thus in view of Theorem
\ref{theorem 71} and Lemma \ref{lemma 71padic}, a proof of Theorem
\ref{theorem 712} can be easily obtained by mimicking the proofs of
Theorems \ref{theorem 17} \& \ref{theorem 72}.

\subsection{Quasi-independence on average method}
One may ask in the $p$-adic case whether there exists a non-negative
function $\psi:\mathbb{N}\rightarrow\mathbb{R}$ for which one cannot
use the quasi-independence on average method to prove
$\mu_p(W_p(\psi))=1$. We will give a conditional but almost best
possible answer to this question: If the $p$-adic version of the
Duffin-Schaeffer conjecture is true, then the answer is NO. To this
purpose
 we first recall several facts proved by
Haynes (\cite{Haynes p-adic}):

\begin{itemize}
\item H1: $\sum\limits_{n\in\mathbb{N}\atop \psi(n)\geq1}\mu_p({\mathcal
K}_n(\psi))=\infty$ implies $\mu_p(W_p(\psi))=1$.

\item H2: $\mu_p(W_p(t\psi))\in\{0,1\}$ is independent of $t>0$.

\item H3: If $p|n$, then ${\mathcal K}_n(\psi)=\emptyset$ or
$\mathbb{Z}_p$.

\item H4: Suppose that $\frac{\psi(n)}{n}$ takes value in the set
$\{0,1,p^{-1},p^{-2},\ldots\}$ and $\psi(n)<\frac{1}{4}$ for all
$n\in\mathbb{N}$. Then for all $m,n\in\mathbb{N}$ with $p\nmid m,n$
we have that\\ $\lambda({\mathcal E}_m(\frac{\psi}{2})\cap{\mathcal
E}_n(\frac{\psi}{2}))\leq\mu_p({\mathcal K}_m(\psi)\cap{\mathcal
K}_n(\psi))\leq\frac{3}{2}\cdot\lambda({\mathcal
E}_m(2\psi)\cap{\mathcal E}_n(2\psi))$.
\end{itemize}
We remark H1 (\cite[Thm. 3 (i)]{Haynes p-adic}) is a
Pollington-Vaughan type theorem (\cite[Thm. 2]{PV2}), while H2
(\cite[Lemma 1]{Haynes p-adic}) is a Cassels-Gallagher type zero-one
law (\cite{Cassels,Gallagher}). The proof of H3 (\cite[lemma
2]{Haynes p-adic}) is trivial, while that of  H4 (\cite[lemma
3]{Haynes p-adic}) is similar to the overlap estimates obtained in
\cite{PV2}.

Now suppose the $p$-adic version of the Duffin-Schaeffer conjecture
is true and let $\psi$ be any non-negative function such that
$\sum_n\mu_p({\mathcal K}_n(\psi))=\infty$. In the following we will
use the quasi-independence on average method to deduce
$\mu_p(W_p(\psi))=1$.
 According to H1$\sim$H3 we may further assume without loss of
generality that
 $\psi<\frac{1}{4}$ and
$\psi(n)=0$ for all $n\in p\mathbb{N}$. We define
\begin{align*}\psi_2(n)=\begin{cases}
\psi(n) & (n\in\Delta_h, S_h(\psi)\leq1) \\
\frac{\psi(n)}{S_h(\psi)}& (n\in\Delta_h, S_h(\psi)>1).
\end{cases}\end{align*}
Obviously,
 $\psi_2\leq\psi<\frac{1}{4}$,
 $\psi_2(n)=0$ for all $n\in p\mathbb{N}$,
 $S_h(\psi_2)\leq1$ for all $h\in\mathbb{N}$,
and $\sum_hS_h(\psi_2)=\infty$. Since we are working in
$\mathbb{Z}_p$, it does not change anything on the $p$-adic side of
things if we round down each of the values taken by the function
$n\mapsto\frac{\psi_2(n)}{n}$ so that
 the range of the function $n\mapsto\frac{\psi_2(n)}{n}$ is contained in  the set
$\{0,p^{-1},p^{-2},\ldots\}$. Hence by H3 and H4 we have for all
$m,n\in\mathbb{N}$ that
\begin{equation}\label{formula 61}\lambda({\mathcal E}_m(\frac{\psi_2}{2})\cap{\mathcal
E}_n(\frac{\psi_2}{2}))\leq\mu_p({\mathcal K}_m(\psi_2)\cap{\mathcal
K}_n(\psi_2))\leq\frac{3}{2}\cdot\lambda({\mathcal
E}_m(2\psi_2)\cap{\mathcal E}_n(2\psi_2)),\end{equation} which
combining Theorem \ref{theorem 712} gives a universal constant $C>0$
such that
\begin{equation}\label{formula 62}CS_h(\psi_2)\leq\mu_p\big(\bigcup_{n\in\Delta_h}{\mathcal
K}_n(\psi_2)\big)\leq3S_h(\psi_2)\ \ \ (h\in\mathbb{N}).
\end{equation} As  $\sum_hS_h(\psi_2)=\infty$,
there exists an integer $i\in\{0,1,2\}$ such that
\[\sum_{h=1}^{\infty}\mu_p\big(\bigcup_{n\in\Delta_{3h+i}}{\mathcal
K}_n(\psi_2)\big)=\infty.\] So we can apply Lemma \ref{lemma 16}
together with the estimates (\ref{formula 2999}), (\ref{formula
61}), (\ref{formula 62}) to get
\begin{align*}
\mu_p(W_p(\psi_2))&\geq\limsup_{N\rightarrow\infty}\frac{\displaystyle\Big(\sum_{h=1}^N\mu_p\big(\bigcup_{n\in\Delta_{3h+i}}{\mathcal
K}_n(\psi_2)\big)\Big)^2}{\displaystyle\sum_{h_1=1}^N\sum_{h_2=1}^N\lambda\Big(\big(\bigcup_{m\in\Delta_{3h_1+i}}{\mathcal
K}_m(\psi_2)\big)\cap\big(\bigcup_{n\in\Delta_{3h_2+i}}{\mathcal
K}_n(\psi_2)\big)\Big)}\\
&\gg\limsup_{N\rightarrow\infty}\frac{\displaystyle\Big(\sum_{h=1}^NS_{3h+i}(\psi_2)\Big)^2}{\displaystyle
\sum_{h=1}^NS_{3h+i}(\psi_2)+2\sum_{1\leq h_1<h_2\leq
N}\sum_{m\in\Delta_{3h_1+i} \atop
n\in\Delta_{3h_2+i}}\lambda\big({\mathcal E}_m(2\psi_2)\cap{\mathcal
E}_n(2\psi_2)\big)}\\
&\gg\limsup_{N\rightarrow\infty}\frac{\displaystyle\Big(\sum_{h=1}^NS_{3h+i}(\psi_2)\Big)^2}{\displaystyle
\sum_{h=1}^NS_{3h+i}(\psi_2)+\Big(\sum_{h=1}^NS_{3h+i}(\psi_2)\Big)^2}>0.
\end{align*}
Consequently, we can use H2 to deduce $\mu_p(W_p(\psi_2))=1$. As
$\psi_2\leq\psi$, $\mu_p(W_p(\psi))=1$. This provides a
quasi-independence on average proof of the claim.

Based on H1$\sim$H3, to study the $p$-adic version of the
Duffin-Schaeffer conjecture one may always assume $\psi<\frac{1}{4}$
and $\psi(n)=0$ for all $n\in p\mathbb{N}$. With these assumptions
we believe the heart of the conjecture  is to establish
\begin{equation}\mu_p\big(\bigcup_{n\in\Delta_h}{\mathcal K}_n(\psi)\big)\asymp
 S_h(\psi)\end{equation} whenever $S_h(\psi)\leq1$.

\section{Diophantine approximation over formal Laurent
series}\label{Laurent}

The study of the $p$-adic version of the Duffin-Schaeffer conjecture
highly resembles that of the classical Duffin-Schaeffer conjecture
for at least both problems deal with non-negative functions from
$\mathbb{N}$ to $\mathbb{R}$. In the formal Laurent series case we
will study non-negative functions from $\mathbb{F}[X]$ to
$\mathbb{R}$.

Throughout this section $Q$ will always be regarded as a monic
polynomial wherever you meet. Recall that $q$ stands for the size of
$\mathbb{F}$.

\subsection{Proof of Theorem \ref{theorem 78}} Ahead of proving Lemma \ref{lemma 79} let us explain the constructin
of the Haar measure $\nu$ on $\mathbb{F}((X^{-1}))$ in a much
straightforward way.
 Any
bijection $\tau:\mathbb{F}\mapsto\{0,1,\ldots,q-1\}$ naturally
induces a map $\widehat{\tau}:\mathbb{F}((X^{-1}))\mapsto\mathbb{R}$
 sending $\sum_ia_iX^i$ to $\sum_i\tau(a_i)q^i.$ A subset $A$ of
$\mathbb{F}((X^{-1}))$ is said to be $\tau$-measurable if
$\widehat{\tau}(A)$ is Lebesgue measurable in $\mathbb{R}$. For any
$\tau$-measurable subset $A$ of $\mathbb{F}((X^{-1}))$, we define
its $\tau$-measure by $\nu_{\tau}(A)=\lambda(\widehat{\tau}(A))$.
For any permutation $\gamma$ on $\{0,1,\ldots,q-1\}$ it is easy to
show that
$\widehat{\gamma}:\sum_ib_iq^i\in\mathbb{R}\rightarrow\sum_i\gamma(b_i)q^i\in\mathbb{R}$
is  measure-preserving on $(\mathbb{R},\lambda)$. This implies that
the concepts of $\tau$-measurable and $\tau$-measure are independent
of the choices of $\tau$. So it brings no confusion to write
$\nu_{\tau}$ simply as $\nu$. The interested readers may easily
verify that $\nu$ is nothing but the unique Haar measure on
$\mathbb{F}((X^{-1}))$ such that $\nu(\mathbb{L})=1$. With this
construction it is easy to see for any $f\in\mathbb{F}((X^{-1}))$
and any $r>0$ that $r\leq\nu(B(f,r))\leq qr$
$(\spadesuit\spadesuit)$.

\begin{lemma}\label{lemma 79} Let $t\geq1$. Then for any
non-negative function $\Psi:\mathbb{F}[X]\rightarrow\mathbb{R}$,
\[\nu_d\big(\bigcup_Q{\mathcal E}_Q(t\Psi)^d\big)\leq q^dt^d\nu_d\big(\bigcup_{Q}{\mathcal E}_Q(\Psi)^d\big).\]
\end{lemma}

\begin{proof}[Proof of Lemma \ref{lemma 79}]
The proof of Lemma \ref{lemma 79} is fully identical to that of
Lemma \ref{lemma 71padic}, but we still provide the details to help
the readers get familiar with the language of formal Laurent series.
By the inner regular property of the Haar measure $\nu_d$
($\bigstar\bigstar$), we may assume without loss of generality that
the support of $\Psi$ is a non-empty set of finite elements. As any
two balls in $\mathbb{F}((X^{-1}))$ can only have either empty
intersection or one is contained in the other
($\clubsuit\clubsuit$), it is rather easy to see that
$\bigcup_{Q}{\mathcal E}_Q(\Psi)^d$ is the union of finitely many
pairwise disjointly supported $d$-dimensional non-empty open cubes
of the following form
\[\bigcup_{Q}{\mathcal E}_Q(\Psi)^d=\bigcup_{i=1}^M
\bigcup_{j=1}^{M_i}\Big(\big(\prod_{k=1}^dB(\frac{P_{i,j,k}}{Q_i},\frac{\Psi(Q_i)}{|Q_i|})\big)\cap\mathbb{L}^d\Big)\
\ \ \ \partial P_{i,j,k}<\partial Q_i,  (P_{i,j,k},Q_i)=1,\] and
\[\bigcup_{Q}{\mathcal E}_Q(t\Psi)^d=\bigcup_{i=1}^M
\bigcup_{j=1}^{M_i}\Big(\big(\prod_{k=1}^dB(\frac{P_{i,j,k}}{Q_i},\frac{t\Psi(Q_i)}{|Q_i|})\big)\cap\mathbb{L}^d\Big).\]
If there exists a pair $(i,j)$ such that $\mathbb{L}^d\subset
\prod_{k=1}^dB(\frac{P_{i,j,k}}{Q_i},\frac{\Psi(Q_i)}{|Q_i|})$, then
$\bigcup_{Q}{\mathcal E}_Q(\Psi)^d=\mathbb{L}^d$ and we need to do
nothing further. Else by the property $\clubsuit\clubsuit$ we can
assume
$\prod_{k=1}^dB(\frac{P_{i,j,k}}{Q_i},\frac{\Psi(Q_i)}{|Q_i|})\subset\mathbb{L}^d$
for all the pairs $(i,j)$. Consequently, by the property
$\spadesuit\spadesuit$ we
have\begin{align*}\nu_d\Big(\bigcup_{Q}{\mathcal
E}_Q(t\Psi)^d\Big)&\leq
\sum_{i=1}^M\sum_{j=1}^{M_i}\nu_d\Big(\prod_{k=1}^dB(\frac{P_{i,k}}{Q_i},\frac{t\Psi(Q_i)}{|Q_i|})\Big)\\
&\leq q^d
t^d\sum_{i=1}^M\sum_{j=1}^{M_i}\nu_d\Big(\prod_{k=1}^dB(\frac{P_{i,k}}{Q_i},\frac{\Psi(Q_i)}{|Q_i|})\Big)=q^dt^d\nu_d\Big(\bigcup_{Q}{\mathcal
E}_Q(\Psi)^d\Big).\end{align*} This finishes the proof of Lemma
\ref{lemma 79}.
\end{proof}

In the formal Laurent series case we do have the Gallagher type
zero-one law (\cite[Thm. 1]{Inoue}). Thus in view of Theorem
\ref{theorem 71} and Lemma \ref{lemma 79}, a proof of Theorem
\ref{theorem 78} can be easily obtained by mimicking the proofs of
Theorems \ref{theorem 17} \& \ref{theorem 72}.

\subsection{Proof of Theorem \ref{theorem 114new}}

\begin{lemma}\label{lemma 72} $\nu_d({\mathcal H}^{(d)}(t\Psi))\in\{0,1\}$ is independent of $t>0$.
\end{lemma}

This lemma is a Cassels-Gallagher type zero-one law. In the
classical case the author established various zero-one laws
((\cite[Theorems 3.1, 3.2, 3.3, 4.1 \& 4.3]{Li1})) via the cross
fibering principle (\cite[Thm. 3]{BHV multi}) of Beresnevich,
Haynes, Velani as well as a multi-purpose
 Cassels-Gallagher type theorem (\cite[Lemma
2.1]{Li1}). In the formal Laurent series case the key to the proof
of Lemma \ref{lemma 72} is the following concept and generalization
of \cite[Lemma 2.1]{Li1}:

\begin{definition}\label{defn 73} For any monic $Q\in\mathbb{F}[X]$, let $\omega(Q)$ be a fixed
non-empty subset of divisors of $Q$. For any non-negative function
$\Psi:\mathbb{F}[X]\rightarrow\mathbb{R}$ we denote by ${\mathcal
H}(\omega,\Psi)$ the set of $f\in\mathbb{L}$ for which
$|Qf-P|<\Psi(Q)$ holds for infinitely many triples
$(Q,P,R)\in\mathbb{F}[X]^3$ with $\partial P<\partial Q$ and $P$
being coprime to some $R\in\omega(Q)$.
\end{definition}

\begin{lemma}\label{lemma 74}
$\nu({\mathcal H}(\omega,tM))\in\{0,1\}$ is independent of $t>0$.
\end{lemma}

The proof of Lemma \ref{lemma 74} is similar to those of  \cite[Thm.
4]{InoueNakada} and \cite[Lemma 2.1]{Li1} with suitable
modifications, and we leave the details to the interested readers.
Now we can give a proof of Lemma \ref{lemma 72} and will only deal
with the case $d=2$ for the sake of simplicity. All the other cases
are left to the readers to check in a similar way.

\textsc{Proof of Lemma \ref{lemma 72}} ($d=2$): For any $s,t>0$, we
denote by ${\mathcal H}_{s,t}(\Psi)$ the set of
$(f,g)\in\mathbb{L}^2$ for which
\begin{equation}\label{formula 71}|f-\frac{P_1}{Q}|<\frac{s\Psi(Q)}{|Q|} \ \mbox{and}\
|g-\frac{P_2}{Q}|<\frac{t\Psi(Q)}{|Q|}\end{equation} for infinitely
many triples $(Q,P_1,P_2)\in\mathbb{F}[X]^3$ with $\partial P_1,
\partial P_2<\partial Q$, $(Q,P_1,P_2)=1$. Obviously,
$\nu_2({\mathcal H}_{s,t}(\Psi))=\nu_2({\mathcal H}_{t,s}(\Psi))$
($\maltese$). We decompose ${\mathcal H}_{s,t}(\Psi)$
 as disjoint unions
\begin{equation}\cup_{g\in\mathbb{L}}{\mathcal H}_{g}(s\Psi,t\Psi)\times\{g\},\end{equation}
where ${\mathcal H}_{g}(s\Psi,t\Psi)$ denotes the set of
$f\in\mathbb{L}$ for which (\ref{formula 71}) holds for infinitely
many triples $(Q,P_1,P_2)\in\mathbb{F}[X]^3$ with $\partial P_1,
\partial P_2<\partial Q$, $(Q,P_1,P_2)=1$.
 For any monic $Q\in\mathbb{F}[X]$, we denote
\[{\omega}_t(Q)=\{Q\}\cup\{(Q,P_2):P_2\in\mathbb{F}[X], \partial P_2<\partial Q, |g-\frac{P_2}{Q}|<\frac{t\Psi(Q)}{|Q|}\}.\]
By Definition \ref{defn 73}  it is easy to see that ${\mathcal
H}_{g}(s\Psi,t\Psi)={\mathcal H}(\omega_t,s\Psi)$. Hence by Fubini's
theorem, Lemma \ref{lemma 74} and $\maltese$, $\nu_2({\mathcal
H}_{s,t}(\Psi))$ is independent of $s,t>0$. On the other hand, we
note from Lemma \ref{lemma 74} that each fiber ${\mathcal
H}_{g}(s\Psi,t\Psi)$ of ${\mathcal H}_{s,t}(\Psi)$ with horizontal
direction has $\nu$-measure either 0 or 1. In a similar way one can
show that each fiber of ${\mathcal H}_{s,t}(\Psi)$ with vertical
direction also has $\nu$-measure either 0 or 1. Consequently,
$\nu_2({\mathcal H}_{s,t}(\Psi))\in\{0,1\}$ follows the cross
fibering principle \cite[Thm. 3]{BHV multi}. This suffices to finish
the proof of Lemma \ref{lemma 74} as we have ${\mathcal
H}^{(2)}(t\Psi)={\mathcal H}_{t,t}(\Psi)$.

\begin{lemma}\label{lemma 75} Suppose $d\geq2$. Then
 $\nu_d({\mathcal H}_Q^{(d)}(\Psi))\geq
\frac{3}{16}\min\{\Psi(Q)^d,1\}$.
\end{lemma}

This lemma can be regarded  as either a Gallagher type estimate
(\cite[formla (9)]{GallagherHigher}, see also \cite{BV higher
dimensions}) or a Pollington-Vaughan type estimate (\ref{formula
141414}). To give a proof we may assume without loss of generality
that $\Psi(Q)<1$, and suppose this is the case. Thus ${\mathcal
H}_Q^{(d)}(\Psi)$ is the union of pairwise disjointly supported
$d$-dimensional open cubes of the following form
\[{\mathcal
H}_Q^{(d)}(\Psi)=\bigcup_{{P_i\in\mathbb{F}[X]\atop
\partial P_i<\partial Q}\atop
(P_1,P_2,\cdots,P_d,Q)=1}\prod_{i=1}^dB(\frac{P_i}{Q},\frac{\Psi(Q)}{|Q|}).\]
By the property $\spadesuit\spadesuit$, we have
\[\nu_d({\mathcal
H}_Q^{(d)}(\Psi))\geq(\frac{\Psi(Q)}{|Q|})^d\cdot\Theta^{(d)}(Q),\]
where $\Theta^{(d)}(Q)$ denotes the size of the set of
$\mathbf{P}=(P_1,\ldots,P_d)\in\mathbb{F}[X]^d$ such that $\partial
\mathbf{P}=\max_i\partial P_i<\partial Q$,
$(\mathbf{P},Q)=(P_1,P_2,\cdots,P_d,Q)=1$. With the help of the
M\"{o}bius function defined by
$$
\mu(Q)= \left\{
  \begin{array}{ll}
  1, & \partial Q=0\\
  (-1)^k,      &  \mbox{if}\ Q\ \mbox{is the product of}\ k\  \mbox{distinct monic irreducible polynomials}\\
  0,   & \mbox{if}\ Q\ \mbox{is divisible by the square of an irreducible polynomial}\\
  \end{array}
  \right.
$$
and one of its fundamental properties $\sum_{R|Q}\mu(R)=0$ whenever
$\partial Q\geq1$, where $\sum_{R|Q}$ denotes the sum over all monic
divisors $R$ of $Q$, we have
\begin{align*}
\Theta^{(d)}(Q)&=\sum_{\partial\mathbf{P}<\partial
Q}\sum_{R|(\mathbf{P},Q)}\mu(R)=\sum_{R|Q}\Big(\mu(R)\cdot\sum_{\mathbf{P}:
R|\mathbf{P},
\partial\mathbf{P}<\partial Q}1\Big)\\
&=\sum_{R|Q}\mu(R)\cdot(\frac{|Q|}{|R|})^d=|Q|^d\cdot\sum_{R|Q}\frac{\mu(R)}{|R|^d}.
\end{align*}
Consequently, $\nu_d({\mathcal
H}_Q^{(d)}(\Psi))\geq\Psi(Q)^d\cdot\sum_{R|Q}\frac{\mu(R)}{|R|^d}.$
Now we suppose $d\geq2$ and have two cases to consider.

Case 1: Suppose $q^{d-1}\geq3$. In this case we have
\[\sum_{R|Q}\frac{\mu(R)}{|R|^d}\geq1-\sum_{k=1}^{\infty}
\sum_{R \ \mbox{is monic},\ \partial
R=k}\frac{1}{q^{kd}}\geq1-\sum_{k=1}^{\infty}\frac{q^k}{q^{kd}}\geq\frac{1}{2}.\]

Case 2: Suppose $q=d=2$. In this case we have
\[\sum_{R|Q}\frac{\mu(R)}{|R|^2}\geq1-\frac{2^1}{2^2}-\frac{2^2-3}{2^4}-\frac{2^3}{2^6}-\frac{2^4}{2^8}-\cdots=\frac{3}{16},\]
where $2^2-3$ comes from the contribution made only by $X^2+X+1$ as
it is the sole second order element whose M\"{o}bius value is
negative. This suffices to conclude the proof of Lemma \ref{lemma
75}.

\begin{lemma}\label{lemma 76} For any $z_1,z_2\in \mathbb{Z}$ and
any
nonzero element $g\in\mathbb{F}((X^{-1}))$,
\[\sum_{\mathbf{P}\in\mathbb{F}[X]^d\backslash\{\mathbf{0}\}}\nu_d\big(\mathbb{L}_{z_1}^d\cap(\mathbb{L}_{z_2}^d+g\mathbf{P})\big)
\leq\frac{\nu_d(\mathbb{L}_{z_1}^d)\cdot\nu_d(\mathbb{L}_{z_2}^d)}{|g|^d},\]
where $\mathbb{L}_z$ denotes the set of elements of
$\mathbb{F}((X^{-1}))$ with degrees less than $z$.
\end{lemma}

\begin{proof}
Without loss of generality we may assume that $z_1\geq z_2$. If
$z_1\leq \partial g$, then we have nothing to prove as the left hand
side of the desired inequality is zero. So we can assume  $z_1>
\partial g$. In this case for any $\mathbf{P}\in\mathbb{F}[X]^d$
with
$\mathbb{L}_{z_1}^d\cap(\mathbb{L}_{z_2}^d+g\mathbf{P})\neq\emptyset$,
one must have $\partial\mathbf{P}<z_1-\partial g$. This implies
\[\sum_{\mathbf{P}\in\mathbb{F}[X]^d\backslash\{\mathbf{0}\}}\nu_d
\big(\mathbb{L}_{z_1}^d\cap(\mathbb{L}_{z_2}^d+g\mathbf{P})\big)\leq(q^{z_1-\partial
g})^d\cdot\nu_d(\mathbb{L}_{z_2}^d).\] This finishes the proof of
Lemma \ref{lemma 76}.
\end{proof}

Proof of Theorem \ref{theorem 114new}: According to Lemmas
\ref{lemma 72}, \ref{lemma 75} and the first Borel-Cantelli lemma,
we may assume without loss of generality that
$\Psi(Q)\in\{q^{-1},q^{-2},q^{-3},\ldots\}$ for all
$Q\in\mathbb{F}[X]$.
 By  Lemma \ref{lemma 72} and the second Borel-Cantelli lemma it suffices to establish the quasi-independence property for
${\mathcal H}_Q^{(d)}(\Psi)$ and ${\mathcal H}_{Q'}^{(d)}(\Psi)$,
where $Q,Q'$ are any two distinct monic elements of $\mathbb{F}[X]$.
To this aim, we first denote $\mathbb{U}(s)=\mathbb{L}^d_{\log_qs}$,
then rewrite ${\mathcal H}_Q^{(d)}(\Psi)$ as
\[
{\mathcal H}_Q^{(d)}(\Psi)=\bigcup_{\partial\mathbf{P}<\partial
Q\atop
(\mathbf{P},Q)=1}\big(\mathbb{U}(\frac{\Psi(Q)}{|Q|})+\frac{\mathbf{P}}{Q}\big).\]
Thus
\[
\nu_d\big({\mathcal H}_Q^{(d)}(\Psi)\cap{\mathcal
H}_{Q'}^{(d)}(\Psi)\big)\leq\sum_{\partial\mathbf{P}<\partial Q\atop
(\mathbf{P},Q)=1} \sum_{\partial\mathbf{P}'<\partial Q'\atop
(\mathbf{P}',Q')=1}\nu_d\Big(\mathbb{U}(\frac{\Psi(Q)}{|Q|})\cap
\big(\mathbb{U}(\frac{\Psi(Q')}{|Q'|})+\frac{\mathbf{P}'}{Q'}-\frac{\mathbf{P}}{Q}\big)\Big).
\]
It is easy to verify that
$\frac{\mathbf{P}'}{Q'}-\frac{\mathbf{P}}{Q}$ is always  a nonzero
element of $\mathbb{F}((X^{-1}))^d$ whenever
$\partial\mathbf{P}<\partial Q,
\partial\mathbf{P}'<\partial Q', (\mathbf{P},Q)=(\mathbf{P}',Q')=1$, and
if we further have
$\frac{\mathbf{P}'}{Q'}-\frac{\mathbf{P}}{Q}=\frac{\mathbf{R}'}{Q'}-\frac{\mathbf{R}}{Q},$
$\partial\mathbf{R}<\partial Q, \partial\mathbf{R}'<\partial Q',
(\mathbf{R},Q)=(\mathbf{R}',Q')=1$, then
\[\partial(\mathbf{P}-\mathbf{R})<\partial Q\ \mbox{and}\ \frac{Q}{(Q,Q')}|(\mathbf{P}-\mathbf{R}).\]
These facts imply that every fixed (nonzero) element of the form
$\frac{\mathbf{P}'}{Q'}-\frac{\mathbf{P}}{Q}$ can be repeated at
most $|(Q,Q')|^d$ times. Consequently, by Lemmas \ref{lemma 75} and
\ref{lemma 76} we have
\begin{align*}
\nu_d\big({\mathcal H}_Q^{(d)}(\Psi)\cap{\mathcal
H}_{Q'}^{(d)}(\Psi)\big)&\leq
|(Q,Q')|^d\cdot\frac{(\frac{\Psi(Q)}{|Q|})^d\cdot
(\frac{\Psi(Q')}{|Q'|})^d}{|\frac{(Q,Q')}{QQ'}|^d}=\Psi(Q)^d\cdot\Psi(Q')^d\\
&\leq \frac{256}{9}\cdot \nu_d({\mathcal
H}_Q^{(d)}(\Psi))\cdot\nu_d({\mathcal H}_{Q'}^{(d)}(\Psi)).
\end{align*}
This finishes the whole proof of Theorem \ref{theorem 114new}.

\section{Weighted second Borel-Cantelli lemma }\label{App BC}

 This section is devoted to the proof of Theorem
\ref{theorem 19}. First let us explain why Theorem \ref{theorem 19}
is a generalization of the Beresnevich-Harman-Haynes-Velani theorem
(\cite[Thm. 2]{BHHV}) which claims $\lambda(W(\psi))=1$ if for some
$c>0$,
\begin{equation}\label{formula 81}
\sum_{n=16}^{\infty}\frac{\varphi(n)\psi(n)}{n\exp(c(\log\log
n)(\log\log\log n))}=\infty.
\end{equation}
Suppose (\ref{formula 81}) is true. Obviously, there exists a
$c_1>0$ such that (see the second section for the meanings of
$\Delta_h,S_h,B_h,Q_h,R_h$)
\[\sum_{h\in\mathbb{N}}\frac{S_h}{\exp(c_1h\log h)}=\infty,\]
from which we can deduce a $c_2>0$ and an infinite subset $H$ of
$\mathbb{N}$ such that for all $h\in H$, $S_h\geq \exp(c_2h\log h)$.
In fact, we can define $h\in H$ by
\[\frac{S_h}{\exp(c_1h\log h)}\geq\frac{1}{h^2}.\]
Thus $\lambda(W(\psi))=1$ follows easily from Theorem \ref{theorem
19}.

Next, let us explain how will we make use of the weighted version
Lemma \ref{lemma 17} in the proof of Theorem \ref{theorem 19}.
According to the discussions in the last paragraph of the second
section, we can assume $P(m,n)\ll1$ for any $m,n$ lying in any
corresponding distinct blocks $\Delta_{h_1},\Delta_{h_2}$. By
appealing to a theorem  of Pollington and Vaughan (\cite[Theorem
2]{PV2}), we may also assume without loss of generality that
$\psi(n)\leq1/2$ for all $n\in\mathbb{N}$. This means
$\lambda({\mathcal E}_n)=2\frac{\psi(n)\varphi(n)}{n}$ for all
$n\in\mathbb{N}$.
 Fix
arbitrarily $\omega_{\Delta_h}\in[0,1]$ such that
$\sum_h\omega_{\Delta_h}S_h=\infty$. With
$\omega_n\triangleq\omega_{\Delta_h}$ $(n\in\Delta_h)$ and
${\mathcal A}_n\triangleq{\mathcal E}_n$ $(n\geq5)$ we can apply
Lemma \ref{lemma 17} to get
\begin{align}\lambda(W(\psi))\gg\limsup_{N\rightarrow\infty}\frac{\displaystyle\big(\sum_{h=1}^N\omega_{\Delta_h}S_h\big)^2}{\displaystyle
\sum_{h=1}^N\omega_{\Delta_h}^2S_h+\sum_{h=1}^N\omega_{\Delta_h}^2Q_h+\big(\sum_{h=1}^N\omega_{\Delta_h}S_h\big)^2}.\end{align}
Thus if one can show that
\begin{align}\label{formula 52}\sum_{h=1}^N\omega_{\Delta_h}^2Q_h\ll\big(\sum_{h=1}^N\omega_{\Delta_h}S_h\big)^2\end{align}
for infinitely many $N$, then $\lambda(W(\psi))=1$ follows from
Gallagher's zero-one law.


\textsc{Proof of Theorem \ref{theorem 19}}: Suppose we have
(\ref{formula 111111111111}), which is equivalent to
\begin{equation}\label{formula 53}\sum_{h:S_h\geq e^e}\frac{\log S_h}{h\cdot \log\log S_h}=\infty.\end{equation}
 Our purpose is to show that $\lambda(W(\psi))=1$. If there exist infinitely many $h\in\mathbb{N}$
such that $S_h\geq \exp(h\log h)$, then we can apply the
aforementioned Beresnevich-Harman-Haynes-Velani theorem to deduce
$\lambda(W(\psi))=1$. Thus we can assume without loss of generality
that
\begin{equation}\label{formula 54444}S_h\leq
\exp(h\log h)\ \ (h\in\mathbb{N}).\end{equation} For any
$h\in\mathbb{N}$ with $S_h\geq e^e$, let ${\mathcal D}_j^{(h)}$
$(j\geq-1)$ denote the collection of pairs
$(m,n)\in\Delta_h\times\Delta_h$, $m\neq n$, such that $e^j\leq
D(m,n)<e^{j+1}$. Denote
\[A_j^{(h)}=A_j^{(h)}(\psi)=\sum_{(m,n)\in {\mathcal D}_j^{(h)}}\psi(m)\psi(n)\frac{\varphi(m)}{m}\frac{\varphi(n)}{n}.\]
 Define two functions $f^{(h)}$, $g^{(h)}$ on the set of
integers $\mathbb{Z}$ respectively by
\begin{align*}f^{(h)}(j)&=\begin{cases}
A_j^{(h)} & (j\geq-1) \\
0 & (j<-1),
\end{cases}\\
g^{(h)}(j)&=\begin{cases}
\frac{h}{-j+2} & (-\frac{\log S_h}{\log\log S_h}\leq j\leq1) \\
0 &  (j>1\ \mbox{or}\ j<-\frac{\log S_h}{\log\log S_h}).
\end{cases}\end{align*}
The convolution of $f^{(h)}$ and $g^{(h)}$ is defined usually as
\[(f^{(h)}\ast g^{(h)})(k)=\sum_{{j\in\mathbb{Z}}}f^{(h)}(j)g^{(h)}(k-j)\ \ \ (k\in\mathbb{Z}).\]
For simplicity we denote $y_h=\frac{\log S_h}{\log\log S_h}$. Note
it is easy to deduce from  $e^e\leq S_h\leq\exp(h\log h)$ that
$y_h\leq h$. With these preparations we have
\begin{align*}
\sum_{k:0\leq k\leq \log
S_h}\sum_{j=k-1}^{k+y_h}\frac{h}{j-k+2}A_j^{(h)}&\leq\|f^{(h)}\ast
g^{(h)}\|_{l^1(\mathbb{Z})}\\
&\leq\|f^{(h)}\|_{l^1(\mathbb{Z})}\cdot
\|g^{(h)}\|_{l^1(\mathbb{Z})}\ll S_h^2\cdot h\cdot\log\log S_h.
\end{align*}
Thus there exists a non-negative integer
\begin{equation}\label{formula 55555}k_h\leq \log S_h\end{equation}
such that \[\sum_{j=k_h-1}^{k_h+y_h}\frac{h}{j-k_h+2}A_j^{(h)}\ll
\frac{S_h^2\cdot h\cdot \log\log S_h}{\log S_h},\] which is
equivalent to
\begin{equation}\label{formula 565656}
\sum_{j=k_h-1}^{k_h+y_h}\frac{h}{j-k_h+2}A_j^{(h)}(\frac{\psi}{e^{k_h}})\ll
S_h(\frac{\psi}{e^{k_h}})^2\cdot\frac{h\cdot\log\log S_h}{\log
S_h}.\end{equation} On the other hand,
\begin{equation}\label{formula 575757}\sum_{j\geq
k_h+y_h+1}\frac{h}{j-k_h+2}A_j^{(h)}(\frac{\psi}{e^{k_h}})\ll
\frac{h}{y_h}\cdot S_h(\frac{\psi}{e^{k_h}})^2=
S_h(\frac{\psi}{e^{k_h}})^2\cdot\frac{h\cdot\log\log S_h}{\log
S_h}.\end{equation} Combining (\ref{formula 565656}) and
(\ref{formula 575757}) gives \[\sum_{j\geq
k_h-1}\frac{h}{j-k_h+2}A_j^{(h)}(\frac{\psi}{e^{k_h}})\ll
S_h(\frac{\psi}{e^{k_h}})^2\cdot\frac{h\cdot\log\log S_h}{\log
S_h},\] which followed by applying (\ref{formula 2222}) and
(\ref{formula 25})  gives
\begin{align}\label{formula 566666}
R_h(\frac{\psi}{e^{k_h}})\ll\frac{h\cdot \log\log S_ h}{\log S_h}.
\end{align}
We now define
\begin{align*}
\overline{\psi}(n)&=\begin{cases}
\frac{\displaystyle\psi(n)}{\displaystyle e^{k_h}} & (n\in\Delta_h,  S_h\geq e^e) \\
0 & (\mbox{otherwise}),
\end{cases}\\
\omega_{\Delta_h}&=\begin{cases}
\frac{\displaystyle 1}{\displaystyle S_h(\overline{\psi})}\cdot\frac{\log S_h}{h\cdot\log\log S_h} & (S_h\geq e^e) \\
0 & (\mbox{otherwise}).
\end{cases}
\end{align*}
 By (\ref{formula 53})$\sim$(\ref{formula 55555}) we see that
 $\omega_{\Delta_h}\leq1$ and
$\sum_h\omega_{\Delta_h}S_h(\overline{\psi})=\infty$. By
(\ref{formula 566666}) we have
\begin{align}\label{formula 577}
\sum_{(N)}\omega_{\Delta_h}^2
Q_h(\overline{\psi})\ll\sum_{(N)}\frac{\log S_h}{h\cdot\log\log S_h}
\ll\big(\sum_{(N)} \frac{\log S_h}{h\cdot\log\log
S_h}\big)^2=\big(\sum_{(N)}\omega_{\Delta_h}S_h(\overline{\psi})\big)^2\end{align}
for large enough $N$, where $\sum_{(N)}$ denotes the sum over all
$h\in[1,N]$ with $S_h\geq e^e$. It follows from the definition of
$\omega_{\Delta_h}$ and (\ref{formula 577}) that for large enough
$N$,
\begin{align}\label{formula 58}
\sum_{h=1}^N\omega_{\Delta_h}^2\cdot
Q_h(\overline{\psi})\ll\big(\sum_{h=1}^N\omega_{\Delta_h}S_h(\overline{\psi})\big)^2,
\end{align}
which verifies (\ref{formula 52}) for the function
$\overline{\psi}$. So we get $\lambda(W(\overline{\psi}))=1$. Since
$\overline{\psi}\leq\psi$, we immediately have $\lambda(W(\psi))=1$.
This finishes the whole proof of Theorem \ref{theorem 19}.

\begin{remark}
We remark that one can also use the weighted second Borel-Cantelli
lemma to reprove a theorem of the author (\cite{Li}) claiming
$\lambda(W(\psi))=1$ if
\[\sum_{n=1}^{\infty}\psi(n)^{1+\epsilon}\cdot\frac{\varphi(n)}{n}=\infty\]
for some $\epsilon>0$, where $\psi$ is any prescribed bounded
non-negative function. Obviously, we may assume $\epsilon<0.5$. By
\cite[Lemma 5]{Harman} there exists a sequence of distinct integers
$\{n_k\}$ such that $\psi(n_k)^{\epsilon}\cdot\lambda({\mathcal
E}_{n_k}(\psi))<1/k$ for all $k\in\mathbb{N}$ and
\[
\sum_{k=1}^{\infty}\psi(n_k)^{\epsilon}\cdot\lambda({\mathcal
E}_{n_k}(\psi))=\infty.
\]
We can first apply the weighted second Borel-Cantelli lemma with
$\omega_k\triangleq\psi(n_k)^{\epsilon}$ and ${\mathcal
A}_k\triangleq{\mathcal E}_{n_k}(\psi)$ to give a lower bound for
$\lambda(W(\psi))$, then mimic Harman's proof of \cite[Thm.
1]{Harman} to deduce the positiveness of this lower bound. Finally
by Gallagher's zero-one law, we are done.  The details are left to
the interested readers to verify.
\end{remark}

\section{Further questions}

Based on Theorem \ref{theorem 17}, we know that if the classical
Duffin-Schaeffer conjecture is true, then there exists a universal
constant $C>0$ such that
\[\lambda(\bigcup_{n=1}^{\infty}{\mathcal E}_n(\psi))\geq C\sum_{n=1}^{\infty}\lambda({\mathcal
E}_n(\psi))\] for any non-negative function
$\psi:\mathbb{N}\rightarrow\mathbb{R}$
 with
$\sum_{n=1}^{\infty}\lambda({\mathcal E}_n(\psi))\leq1$. This lower
bound may not well reflect the true value of
$\lambda(\cup_{n=1}^{\infty}{\mathcal E}_n(\psi))$ when
$\sum_{n=1}^{\infty}\lambda({\mathcal E}_n(\psi))$ is sufficiently
large, so we are bold enough to propose

\begin{Conjecture}\label{conjecture 91}
There exists a universal constant $M_d>0$ depending only on
$d\in\mathbb{N}$ such that if
$\sum_{n=1}^{\infty}\lambda_d({\mathcal E}_n(\psi)^d)\geq M_d$, then
$\lambda_d(\cup_{n=1}^{\infty}{\mathcal E}_n(\psi)^d)=1.$
\end{Conjecture}

Conjecture \ref{conjecture 91}  implies the classical
Duffin-Schaeffer and Sprind\u{z}uk conjectures. Taking for granted
that Conjecture \ref{conjecture 91} is true, we can understand  the
quantitative theory pioneered by Schmidt (\cite{Schmidt}) even
better. Let $\psi:\mathbb{N}\rightarrow\mathbb{R}$ be any
non-negative function such that
$\sum_{n=1}^{\infty}\lambda_d({\mathcal E}_n(\psi)^d)=\infty$, and
let $0=N_0<N_1<N_2<N_3<\cdots$ be the unique sequence of integers
such that for all non-negative integers $k$,
\[\sum_{n=N_k+1}^{N_{k+1}-1}\lambda_d({\mathcal E}_n(\psi)^d)<M_d
\leq\sum_{n=N_k+1}^{N_{k+1}}\lambda_d({\mathcal E}_n(\psi)^d)<
M_d+1.\] Note for any $N\in\mathbb{N}$, there exists a unique
$k\in\mathbb{N}\cup\{0\}$ such that $N_k<N\leq N_{k+1}$, from which
we can easily deduce \[kM_d\leq\sum_{n=1}^N\lambda_d({\mathcal
E}_n(\psi)^d)\leq(k+1)(M_d+1).\]
 By the assumed truth of
Conjecture \ref{conjecture 91} we have for almost all
$\mathbf{x}\in(\mathbb{R}/\mathbb{Z})^d$,
\begin{align*}
M(N,\mathbf{x})&\triangleq\sharp\{n\in\mathbb{N}:n\leq N_k,\
\mathbf{x}\in{\mathcal E}_n(\psi)^d\} \geq M(N_k,\mathbf{x})\geq
k\\&\geq\frac{\displaystyle\sum_{n=1}^N\lambda_d({\mathcal
E}_n(\psi)^d)}{\displaystyle M_d+1}-1.\end{align*} Note also
\[\int_{(\mathbb{R}/\mathbb{Z})^d}M(N,\cdot)=\sum_{n=1}^N\lambda_d({\mathcal
E}_n(\psi)^d).\] Consequently,  the above lower bound is best
possible despite some loss of constant.

Conjecture \ref{conjecture 91} might be too strong to hold, so we
 instead propose a weaker version.

 \begin{Conjecture}\label{conjecture 92}
There exists a universal constant $M_{d,\gamma}>0$ depending only on
$d\in\mathbb{N}$ and $\gamma\in(0,1)$ such that if
$\sum_{n=1}^{\infty}\lambda_d({\mathcal E}_n(\psi)^d)\geq
M_{d,\gamma}$, then $\lambda_d(\cup_{n=1}^{\infty}{\mathcal
E}_n(\psi)^d)\geq\gamma.$
\end{Conjecture}

\end{document}